\documentclass[10pt]{article}

\usepackage{amsfonts,amssymb,amsmath,amsthm,indentfirst}
\usepackage{fontenc,enumerate}
\usepackage{graphics}
\usepackage{color,graphicx}
\usepackage{bbm}

\usepackage{latexsym,mathrsfs,enumerate,verbatim,bm}
\usepackage{cite}

\usepackage[colorlinks, linkcolor=blue, citecolor=blue, urlcolor=blue, pagebackref, hypertexnames=false]{hyperref}

\usepackage[light]{antpolt}    \usepackage[T1]{fontenc}

\usepackage{wasysym}
\usepackage{graphics}

\usepackage{color,graphicx}

\DeclareMathAlphabet{\mathpzc}{OT1}{pzc}{m}{it}

\newcommand{\sett}[1]{\left\{   #1   \right\}}

\usepackage[colorinlistoftodos]{todonotes}
\makeatletter
\providecommand\@dotsep{5}
\def\listtodoname{List of Todos}
\def\listoftodos{\@starttoc{tdo}\listtodoname}
\makeatother
\allowdisplaybreaks
\usepackage{xcolor}
\usepackage[pagewise,mathlines]{lineno}

\newcommand{\vertiii}[1]{{\left\vert\kern-0.25ex\left\vert\kern-0.25ex\left\vert #1 
		\right\vert\kern-0.25ex\right\vert\kern-0.25ex\right\vert}}
\numberwithin{equation}{section}
\newtheorem{theorem}{\quad Theorem}[section]
\newtheorem{lemma}[theorem]{\quad Lemma}

\newtheorem{remark}[theorem]{\quad Remark}

\newtheorem{proposition}[theorem]{\quad Proposition}

\newtheorem{definition}{Definition}[section]

\newcommand{\R}{\mathbb{R}}

\newcommand{\dd}{{\rm d}}

\newcommand{\ii}{{\rm i}}

\newcommand{\N}{\mathbb{N}}

\newcommand{\C}{\mathbb{C}}

\newcommand{\al}{\alpha}
\newcommand{\rr}{\mathbb{R}}

\allowdisplaybreaks
  \author{Vicente Alvarez 
  and  Amin Esfahani
  \footnote{  
  E-mail: saesfahani@gmail.com, amin.esfahani@nu.edu.kz} }

\title{Existence of  dipoles    of Klein–Gordon–Zakharov system
	}

\date{}

\begin{document}

	  \maketitle

	\begin{abstract}
 In this paper, we study the long time behavior of solutions of Klein-Gordon-Zakharov system. We show that there exists a  solution with special characteristics, which we shall refer to as a dipole solution, that is, there exists a solution $\vec{u}$ such that 
 	$$\left\|\vec{u}(t)-\sum_{k=1}^{2}\vec{R}_{k}\right\|_{X}
					\to 0 \, \, \text{as}\, \, t\to \infty,$$
                    where $\vec{R}_{k}$ represents a solitary wave for each $k$, with a translation $z_k$ with respect to its position, satisfying that $$|z_1(t)-z_2(t)| \sim 2\log(t)\, \, \text{as} \, \, t\to \infty.$$
 Our approach will initially focus on the spectral analysis of the Hamiltonian operator associated with our system. Subsequently, we aim to establish a coercivity estimate that will allow us to derive conditions ensuring the existence of our solution. It is important to note that, in this problem, our objective is to obtain approximate solutions by solving a final data problem. These approximate solutions will then be used, through uniform estimates and compactness results, to derive the desired conclusions via density arguments.

	\end{abstract}
 
	\section{Introduction}

In this paper, we consider the following system of  Klein–Gordon–Zakharov (KGZ) equations	
 	\begin{equation}\label{system1}
 	    \begin{cases}
      u_{tt}-u_{xx}+u+\al uv+\beta|u|^2u=0,\\
      v_{tt}-v_{xx}=(|u|^2)_{xx},\qquad x,t\in\rr,
  \end{cases}
 	\end{equation}
where $\alpha\in\rr$ and $\beta\neq0$. 
 Various forms of the Zakharov equations play a central role in the theoretical study of strong Langmuir turbulence in plasma physics (see Guo \cite{guo}, Ozawa, Tsutaya, and Tsutsumi \cite{ott,ott-2}, Thornhill and ter Haar \cite{Tt}, and Tsutaya \cite{Tsu}). A number of works have focused on the Klein-Gordon-Zakharov system \eqref{system1} with $\beta = 0$, which models the interaction between Langmuir waves and ion sound waves in plasma. In our notation, $u$ corresponds to the fast-scale component of the electric field, while $v$ represents the ion density fluctuation \cite{zakh}.

The earliest local well-posedness result for this system with $\beta=0$ appears in Ginibre and Velo \cite{gv} and Ozawa–Tsutsumi \cite{ott}, where a fixed-point argument in a Strichartz-type space yields solutions that are only conditionally unique. Later, Masmoudi and Nakanishi \cite{mana} established unconditional uniqueness, i.e., uniqueness in a larger energy space under stronger regularity assumptions. Important developments in the late 1990s led to global existence results for small solutions of \eqref{system1}, relying on conservation laws that are effective only for small initial data \cite{ott,ott-2}. Notably, the analysis heavily depends on the dimension and differing propagation speeds.

More recently, Hadžić, Shatah, and Snelson \cite{hss2013} showed that in the case $\beta = 0$, the system \eqref{system1} is locally well-posed in $H^s \times H^{s-1} \times H^{s-1} \times H^{s-2}$ for $s > 1/2$, with a Lipschitz continuous flow map.

Regarding the dynamics of standing waves, it was shown in \cite{hss2013} that for $\beta = 0$, the system admits solitary wave solutions of the form $(u(x,t), v(x,t)) = (\exp(i \omega t)\, \phi_\omega(x), \psi(x))$, where $\phi$ and $\psi$ are either real-valued periodic functions with fixed fundamental period $L$, or localized functions vanishing at infinity. In the latter case, it holds that $\psi = -\phi^2$ and
\[
\phi(x) = \sqrt{2(1 - \omega^2)}\; \mathrm{sech}\left(\sqrt{1 - \omega^2}\; x\right).
\]
Using a refined variational approach, Gan \cite{gan}, Gan and Zhang \cite{gabzhang} (see also \cite{ggz}), and later Ohta and Odorova \cite{ohtaT}, established the existence of ground state standing waves for \eqref{system1} when $\beta = 0$, and analyzed their instability using the potential well method of Payne and Sattinger together with Levine’s concavity argument in higher dimensions. In one spatial dimension, Lin \cite{lin} studied orbital stability and showed that such waves are orbitally unstable if $\omega < 1/\sqrt{2}$, and stable for $1 > \omega > 1/\sqrt{2}$. In the critical case $\omega = 1/\sqrt{2}$, instability was demonstrated in \cite{yin} via a modulation analysis combined with a virial identity.

Beyond ground states, an important class of solutions consists of sign-changing or \emph{dipole-type} solitary waves, which are nonlinear bound states exhibiting spatially antisymmetric profiles or internal nodes. In the present paper, we aim to study the dynamics of two solitary (ground state) waves solutions of \eqref{system1}. 

Such solutions are often constructed as excited states through constrained variational methods or bifurcation techniques \cite{dipole_solitons,dipole_nls}. In scalar models like the nonlinear Schrödinger (NLS) or Korteweg--de Vries (KdV) equations, dipole solutions are characterized by a pair of localized peaks with opposite signs, separated by a distance that, in many regimes, grows logarithmically with respect to a small parameter related to the frequency or energy of the system. Specifically, if $\varepsilon$ denotes such a small parameter, the optimal interaction distance between positive and negative peaks typically satisfies
\[
|x_+ - x_-| \sim C \log\left( \frac{1}{\varepsilon} \right),
\]
reflecting the delicate energy balance required to maintain such configurations \cite{sulem}. In coupled systems such as the Zakharov model, constructing dipole-type solutions is more subtle due to the interplay between the Schrödinger and wave components, and the presence of different propagation speeds. Nevertheless, analogous structures can exist, where $u$ is antisymmetric and $v$ exhibits compatible spatial decay. These configurations may be seen as nonlinear dipoles with out-of-phase behavior between the components. The logarithmic separation phenomenon persists in this context, driven by exponentially decaying interaction tails and constrained by the variational structure of the system.

A natural extension of this theory involves the study of multi-soliton or multi-dipole solutions, especially in regimes of either weak or strong interaction. Historically, such multi-solitary wave solutions were first constructed for integrable equations like the Korteweg–de Vries (KdV) and one-dimensional cubic NLS equations \cite{miura,zakhs}. In these integrable cases, multi-solitons are global-in-time solutions that asymptotically decompose into a sum of individual solitons as $t \to \pm\infty$. In contrast, for nonintegrable equations, exact multi-soliton solutions are generally not expected. However, asymptotic multi-soliton states can still be constructed in the nonlinear Schrödinger equation, the generalized KdV, and the nonlinear Klein–Gordon equation, where solitary components exhibit weak interaction over long times \cite{cotem,martel-1}.
Inspired by the integrable theory, further developments have led to the construction of multi-soliton solutions in regimes of strong interaction. This has been achieved for the mass-critical NLS \cite{mar-ra}, subcritical and supercritical NLS \cite{ng-2}, and gKdV equations \cite{ng-1}. In these cases, the interaction between solitons is significantly more rigid, with the inter-soliton distances typically scaling like $\log t$, and no alternative separation rates appear to be admissible. A classification result for such regimes is given in \cite{Jendrej}. Extending this rich theory to systems such as the nonlinear Klein-Gordon or Zakharov-type models, where multiple solitary structures interact dynamically, remains an active and challenging direction of current research.

%

%%%%%%%%%%%%%%%%%%%%%%%%%%%%%%%%

To present our results we first express \eqref{system1} in a more appropriate form, we rewrite it as follows
  	\begin{equation}\label{system2}
  	  \begin{cases}
  u_t=-\rho,\\
  \rho_t=-u_{xx}+u+\al uv+\beta|u|^2u,\\
v_t=n_x,\\
      n_{t}-v_{x}=(|u|^2)_{x}.
  \end{cases}  
  	\end{equation}
 Now, before proceeding, we must establish that, in the case of system \eqref{system2}, we refer to solitons as those solutions of the form $\vec{R}=(u,\rho,v,n)$, where
$$
\begin{aligned}
& u(x, t)=e^{\ii \omega t} \Phi_{\omega}(x), \quad v(x, t)=\Psi_{\omega}(x), \\
& \rho(x, t)=e^{\ii \omega t} \rho_{\omega}(x), \quad n(x, t)=\eta_{\omega}(x), 
\end{aligned}
$$
 So, $\vec R$ satisfies
\begin{equation}\label{system3}
  \left\{\begin{array}{l}
\rho_{\omega}=\ii \omega \Phi_{\omega},  \\
-\ii \omega \rho_{\omega}=-\Phi_{\omega}^{\prime \prime}+\Phi_{\omega}+\alpha \Phi_{\omega} \Psi_{\omega}+\beta\left|\Phi_{\omega}\right|^{2} \Phi_{\omega}, \\
0=\eta_{\omega}^{\prime}, \\
0=\Psi_{\omega}^{\prime}+\left(\left|\Phi_{\omega}\right|^{2}\right)'.
\end{array}\right.  
\end{equation}
Which is equivalent to
\begin{equation}\label{system4}
  \left\{\begin{array}{l}
\rho_{\omega}=\ii \omega \Phi_{\omega},  \\
-\Phi_{\omega}^{\prime \prime}+(1-\omega^2)\Phi_{\omega}+\alpha \Phi_{\omega} \Psi_{\omega}+\beta\left|\Phi_{\omega}\right|^{2} \Phi_{\omega}=0, \\
\Psi_{\omega}^{\prime}=-\left(\left|\Phi_{\omega}\right|^{2}\right)'.
\end{array}\right.  
\end{equation}
Hence, $\Phi_\omega$ satisfies
\begin{equation}\label{omegastate}
  -\partial_{xx}\Phi_{\omega}+(1-\omega^2)\Phi_{\omega}+(\beta-\alpha)\Phi_{\omega}^{3}=0,  
\end{equation}
Hence, $$\Phi_{\omega}(x)=\sqrt{1-\omega^2}\Phi(\sqrt{1-\omega^2} x),$$
where $\Phi$ is solution of \begin{equation}\label{states}
    -\partial_{xx}\Phi+\Phi+\mathbbm{m}\Phi^3=0
\end{equation} 
with  $\mathbbm{m}=\beta-\alpha<0$. 
It is not difficult to infer that the solution of system \eqref{system3} are in the form of
\begin{equation}\label{soliton1}
    \Vec{\Phi}_{\omega}=(\Phi_{\omega},i \omega\Phi_{\omega}, -|\Phi_{\omega}|^2,0).
\end{equation}
We denote by solitons the solutions of the form $$\vec{R}_{\omega}:=T(\omega ) \Vec{\Phi}_{\omega},$$ where $T(\theta)(f,g,h,k)=(e^{\ii\theta t }f,e^{\ii\theta t }g, h,k).$
 The ground state $\Phi$ solution of   \eqref{states} can be written as $\Phi(x)=q(|x|)$, where $q(x)=\sqrt{\frac{2}{-\mathbbm{m}}}{\rm sech}(x)$. Equivalently, $\Phi_{\omega}(x)=q_{\omega}(|x|)=\sqrt{1-\omega^2}q(\sqrt{1-\omega^2}|x|)$ with $q>0$ satisfying the ODE,
\begin{equation}\label{ode}
-q^{\prime \prime}+(1-\omega^2)q+mq^{3}
=0, \quad q^{\prime}(0)=0, \quad \lim _{r \rightarrow+\infty} q(r)=0 .
\end{equation} 
 
 It is easy to see that 
\begin{equation}\label{ground}
\left|q_{\omega}(r)-\kappa  e^{-(1-\omega^2)^{1/2}r}\right|+\left|q_{\omega}^{\prime}(r)+\kappa  e^{-(1-\omega^2)^{1/2}r}\right| \lesssim r^{-1} e^{-(1-\omega^2)^{1/2}r}.
\end{equation}
for some $\kappa>0$ and $r>1$. See also \cite[Section 4.2]{sulem}.
Additionally, we also have the following result (see \cite[Theorem 8.1.1]{cazenave}). 

Once the soliton solutions for system \eqref{system2} have been defined, we proceed to establish our object of study, namely, to ensure the existence of dipole solutions for system \eqref{system1}.

\begin{theorem}\label{2maintheorem}
				For $j\in \{1,2\}$ let us  $|\omega_{j}|\leq 1$.
				%Define 
			%\[
			%	\begin{aligned}
			%		\omega_{\star}:=\frac{1}{256}\min \left\{(\lambda_0^j)^{\frac32}\omega_j,1-\omega_{j}^2, \left|\omega_{j}-\omega_{k}\right|:  j, k=1,2, j \neq k\right\}. 
			%	\end{aligned}
		%\]
				If $\omega_{1} \neq \omega_{2}$, then there exist $T_{0} \in \R$ and a solution $\Vec{u}$ for \eqref{system2} defined in $\left[T_{0},+\infty\right)$ such that for all $t \in\left[T_{0},+\infty\right)$ the following estimate holds
				\begin{equation*}\label{2desprin}
					\left\|\vec{u}(t)-\sum_{k=1}^{2}\vec{R}_{k}\right\|_{X}
					\leq t^{-1},
				\end{equation*}
                where $X=H^1(\R)\times H^1(\R)\times L^2(\R)\times L^2(\R)$, $\vec{R}_{k}=T(\omega)\vec{\Phi}_{k}=T(\omega)\vec{\Phi}_{\omega_k}( \cdot- z_{k}(t
                    ))$ as in \eqref{standing}
                and $$|z_{1}(t)-z_{2}(t)|=(1+o(1))\log t  \qquad \mbox{as} \, \, t\to \infty.$$
			\end{theorem}
{Note that, as in \cite{ARYAN},  the Theorem \ref{2maintheorem} establishes  solutions with the same sign that remain at a distance from each other dictated by a logarithmic function. That is, the behavior of the solutions is strongly linked to an ODE of the form$$\ddot{z}(t)=-2e^{-z(t)},$$ whose explicit solution can be described under certain conditions $(z(1),\dot{z}(1))=(0,2)$ and is given by $2\log(t).$}
 
 The proof of the main theorem is divided into two main parts. In the first part, we aim to establish a key coercivity result, which relies on spectral theory techniques applied to the Hamiltonian operator associated with the system under consideration. This coercivity result enables us to impose suitable conditions on certain parameters that must be modulated and orthogonalized to ensure precise control over the associated estimates. 

 A crucial element in this development is the derivation of a bootstrap result for the approximate solutions. Additionally, we introduce a suitable parametrization of specific quantities, allowing us to obtain uniform estimates on relevant operators. These estimates will later be combined to derive the desired conclusions. 

 The second part of the proof is more technical in nature. It focuses on constructing approximate solutions to system \eqref{system2} that satisfy the properties required in the main theorem. Once these approximate solutions are rigorously defined, convergence strategies are employed to establish the existence of a dipole solution in the space $X$ for system \eqref{system1}. 

 Although the proof of Theorem \ref{2maintheorem} is inspired by the approach in \cite{ARYAN}, several modifications are required in our setting, particularly concerning the choice of modulation parameters and the localization strategy. In developing our method, we encounter distinct challenges, primarily because the parameters under consideration differ significantly from those in previous works. Furthermore, we introduce a different operator associated with the system, which necessitates a tailored control strategy. Our approach relies on bounding this operator using prior estimates arising from the interaction between the solitons and the intrinsic structure of the system under study.

The structure of the paper is as follows. In the next section, we present some preliminary results on the solution $\vec\Phi$ and establish a well-posedness result for the Cauchy problem associated with \eqref{system1}.  In Section~\ref{section-coer}, we prove a coercivity property that is crucial to our analysis and ultimately leads to the existence of a dipole solution. We will later establish some results for final data solutions of system \eqref{system2}, and similarly, we will derive uniform estimates that will allow us to obtain the desired result.

            \section{Preliminaries }
In this section, we present some preliminary results and facts about the solitary waves of \eqref{system1}, and we also provide a well-posedness result for the associated Cauchy problem.
            
\begin{proposition}\label{decaimentoquadratico}
		 Consider $\Phi \in H^1(\R) $ a solution of $$-\partial_{xx}\Phi-\Phi+\mathbbm{m}\Phi^{3}=0.$$ 
 Then,

			\begin{enumerate}[(i)]

\item  $\Phi \in C^{2}(\R)$ and $\left|\partial_x^{\beta} \Phi(x)\right|\stackrel{|x|\to \infty}{\longrightarrow } 0$ for all $|\beta| \leq 2$.
				
\item For any  $0<\epsilon <1$,
				\[
				e^{\epsilon|x|} \left(\left|\Phi(x)\right|+\left|\partial_{x} \Phi(x)\right|\right) \in L^{\infty}(\R).
				\]    
			\end{enumerate}	 
		\end{proposition} 
     Let us then consider the following notations $(z_1,z_2,\omega_1,\omega_2)$ with $|\omega_1|<1, \, \, |\omega_2|<1$, 
 $\theta_1, \theta_2 \in \R$  and $|z_1|\geq 1$, $|z_2|\geq 1$, where $$z=z_1-z_2 \, \, \, and \, \, \, \omega=\omega_1-\omega_2.$$ Then, we define the  {standing  waves}  by
 \begin{equation}\label{standing}
    \Phi_{k}=\Phi_{\omega_k}(\cdot-z_k) \, \, \, and \, \, \, \vec{\Phi}_{k} = (\Phi_{k}^{(1)},\Phi_{k}^{(2)},\Phi_{k}^{(3)},\Phi_{k}^{(4)})=(\Phi_{k},i\omega_k\Phi_{k},-\left|\Phi_{k}\right|^2,0).
 \end{equation}
 $$\vec{R}_{k}=T(\omega_k)\vec{\Phi}_{k}\, \, \, \, \, and \, \, \, \,\vec{R}_{k} =(R_{k}^{(1)},R_{k}^{(2)},R_{k}^{(3)},R_{k}^{(4)})=(e^{i\omega_k t }\Phi_{k},i\omega_ke^{i\omega_k t}\Phi_{k},-\left|\Phi_{k}\right|^2,0)$$

\begin{lemma}\label{solitons}
    The following estimates hold for $|z| \gg 1$. For $\omega_1\neq \omega_2$ and $\omega_{\star}=\frac{1}{81}\min\left\{1-\omega_1^2,1-\omega_2^2\right\}$
 
   \begin{enumerate}[(i)]

        \item  For $f:\C\to \R$ such that $f(u)=|u|^2u$,

\begin{equation*}
\int\left|\Phi_{1}^{(j)} \Phi_{2}^{(k)}\right|^{2} \leq C  {e^{-6\sqrt{\omega_{\star}}|z|}} , \quad \int_{\R}\left| \Phi_{j}^{(1)} \partial_x \Phi_{k}^{(1)}\right|^{2} \leq C  {e^{-6\sqrt{\omega_{\star}}|z|}} \quad k,j=1,2.
\end{equation*}
\item $$\left|\left\langle f(R_1^{(1)}),  R_2^{(1)}\right\rangle-c_{1}g_0 q(|z|)\right|\lesssim |z|^{-1} e^{-\sqrt{\omega_{\star}}|z|}.$$
%and
%\begin{equation*}
%\int_{\R}\left|F(\Phi_{1}^{(1)})\Phi_{2}^{(1)})-F\left(\Phi_{1}^{(1)}\right)-F\left(\Phi_{2}^{(1)}\right)-f\left(\Phi_{1}^{(1)}\right) \Phi_{2}^{(1)}-f\left(\Phi_{2}^{(1)}\right) \Phi_{1}^{(1)}\right| \lesssim t^{-\frac{5}{2}} 
%\end{equation*}
% with   $f(u)=|u|^2u$  and $F(u)=\frac{1}{4}|u|^4$. 
  
%\item For any $m>0$,
%\begin{gather*}
%\int_{\R} \abso{\Phi_{1}^{(j)} }  \abso{\Phi_{2}^{(k)}} ^{1+m} \leq C q(|z|) \quad k,j=1,2,3,4.
%\end{gather*}
 
\item  There exists a smooth function $g:[0,+\infty) \rightarrow \mathbb{R}$ such that, for  $r>1$,
\begin{equation}\label{3.5}
\left|g(r)-g_{0} q_{\omega_1}(r)\right| \lesssim r^{-1} e^{-\sqrt{\omega_{\star}}r}
\end{equation}
where
  \[
g_{0}=\frac{1}{c_{1}} \int_{\R} \Phi_{\omega_1}^{3}(x) e^{-x} \mathrm{~d} x>0, \quad c_{1}=\left\|\partial_{x} \Phi_{\omega_1}\right\|_{L^{2}(\R)}^{2} 
\] 
 and
\begin{equation}\label{4}
\left|\left\langle f(R_1^{(1)}+R_2^{(1)})-f(R_1^{(1)})-f(R_2^{(1)}), \partial_{x} R_{1}^{(1)}\right\rangle-c_{1} \frac{z}{|z|} g(|z|)\right| \lesssim e^{-6\sqrt{\omega_{\star}}|z|}.   
\end{equation}
\begin{equation}\label{5}
\left|\left\langle f(R_1^{(1)}+R_2^{(1)})-f(R^{(1)})-f(R_2^{(1)}), \partial_{x} R_{2}^{(1)}\right\rangle+c_{1} \frac{z}{|z|} g(|z|)\right| \lesssim e^{-6\sqrt{\omega_{\star}}|z|}.
\end{equation}
\end{enumerate}
\end{lemma}
\begin{proof}

(i)  Let us consider $j=k=1$, the other cases are similar, from Proposition \ref{decaimentoquadratico} we have,  $|\Phi_{\omega}(y)| \lesssim e^{-\frac{2}{3}(\sqrt{1-\omega^2})|y|}$, and thus for $y=x-z_1$
$$
\begin{aligned}
\int_{\R}\left|\Phi_{1}^{(1)}\Phi_{2}^{(1)}\right|^{2} \mathrm{~d} y =\int_{\R}\left|\Phi_{\omega_1}(x-z_1)\Phi_{\omega_2}(x-z_2)\right|^{2} \mathrm{~d} y &\leq \int_{\R} e^{-\frac{2}{3}\sqrt{1-\omega_1^2}|y|} e^{-\frac{2}{3}\sqrt{1-\omega_2^2}|y+z|} \mathrm{d} y \\
&\leq\int_{\R} e^{-6\sqrt{\omega_{\star}}|y|} e^{-6\sqrt{\omega_{\star}}|y+z|} \mathrm{d} y\\
&\leq e^{-6\sqrt{\omega_{\star}}|z|} \int_{\R} e^{-12\sqrt{\omega_{\star}}|y|} \mathrm{d} y\\& \leq C e^{-6\sqrt{\omega_{\star}}|z|}.    
\end{aligned}
$$
From which the desired result follows. Note that it is possible to modify the value of $\omega_{\star}$ as desired, so that the interaction between the solitons is as small as needed. A similar procedure can be followed in the case of $\partial_{x}\Phi_{1}^{(1)}\partial_{x} \Phi_{2}^{(1)}$.

%Now, by Taylor's expansion, it follows that
%$$
%\left|F\left(\Phi_{1}^{(1)}+\Phi_{2}^{(1)}\right)-F\left(\Phi_{1}^{(1)}\right)-F\left(\Phi_{2}^{(1)}\right)-f\left(\Phi_{1}^{(1)}\right) \Phi_{2}^{(1)}-f\left(\Phi_{2}^{(1)}\right) \Phi_{1}^{(1)}\right| \lesssim\left|\Phi_{1}^{(1)} \Phi_{2}^{(1)}\right|^{\frac{3}{2}}.
%$$
%So, it only remains to proceed as in the previous case to obtain the desired result.
%(ii)  Let $m>0$, then
%$$
%\begin{aligned}
%\int_{\R}\left|\Phi_{1}^{(1)}\right|\left|\Phi_{2}^{(1)}\right|^{1+m} \mathrm{~d} x&=\int_{\R} \Phi_{\omega_1}(y-z) \Phi_{\omega_2}^{1+m}(y) \mathrm{d} y 
 %\\
%& \quad \lesssim q(|z|) \int_{|y|<\frac{3}{4}|z|} e^{|y|} \Phi^{1+m}(y) \mathrm{d} y + e^{-|z|} \int_{|y|>\frac{3}{4}|z|} e^{|y|} \Phi ^{1+m}(y) \mathrm{d} y \\
%&  \quad \lesssim q(|z|).
%\end{aligned}
%$$

(ii)  We claim the following estimate
\begin{equation}\label{1}
\left|\int_{\R} \Phi_{\omega_1}^{3}(y) \Phi_{\omega_2}
(y+z) \mathrm{d} y- c_1g_{0} \kappa e^{-\sqrt{1-\omega_2^2}|z|}\right|\lesssim | z|^{-1} e^{-\sqrt{\omega_{\star}}|z|}. 
\end{equation}

Then, from \eqref{1} and \eqref{ground} it follows that $$\left|\left\langle f(\Phi_1^{(1)}),  \Phi_2^{(1)}\right\rangle-c_{1}g_0 q(|z|)\right|\lesssim |z|^{-1} e^{-\sqrt{\omega_{\star}}|z|}.$$

Then we proceed to test \eqref{1}. First, for $|y|<\frac{3}{4}|z|$ (and so $|y+z| \geq|z|-|y| \geq \frac{1}{4}|z| \gg 1$ ), we have using \eqref{ground},
$$
|\Phi_{\omega_2}(y+z)-\kappa e^{-(1-\omega_2^2)^{1/2}|y+z|}|\lesssim |y+z|^{-1}e^{-(1-\omega_2^2)^{1/2}|y+z|} \lesssim |z|^{-1}e^{-(\omega_{\star})^{1/2}|z|} e^{(\omega_{\star})^{1/2}|y|} 
$$
In particular,
\begin{equation}\label{2}
\left|\int_{|y|<\frac{3}{4}|z|} \Phi_{\omega_1}^{3}(y)\left[\Phi_{\omega_2}(y+z)-\kappa e^{-(1-\omega_2^2)^{1/2}|y+z|}\right] \mathrm{d} y\right| \lesssim |z|^{-1}e^{-(\omega_{\star})^{1/2}|z|}.
\end{equation}
Moreover, for $|y|<\frac{3}{4}|z|$, we have the expansions
$$
\begin{aligned}
 \left||y+z|-|z|-\frac{y \cdot z}{|z|}\right| \lesssim|z|^{-1}|y|^{2},
\end{aligned}
$$
since,  using the identity
\[
|y+z|^2 = |z|^2 + 2 y \cdot z + |y|^2,
\]
we obtain
\[
|y+z| = \sqrt{|z|^2 + 2 y \cdot z + |y|^2}.
\]
We can factor
\[
|y+z| = |z| \sqrt{1 + 2\frac{y \cdot z}{|z|^2} + \frac{|y|^2}{|z|^2}}.
\]
Since \( |y| < \frac{3}{4}|z| \), the argument inside the square root lies in the interval \((\frac{1}{4}, \frac{9}{4})\), where the square root function is smooth, and we can apply a first-order Taylor expansion:
\[
\left| \sqrt{1+\epsilon} - (1 + \frac{1}{2}\epsilon) \right| \lesssim \epsilon^2
\quad \text{for} \quad |\epsilon| \leq 1,
\]
where
\[
\epsilon = 2\frac{y \cdot z}{|z|^2} + \frac{|y|^2}{|z|^2}.
\]
Thus, we have:
\[
|y+z| = |z| \left(1 + \frac{1}{2} \epsilon + O(\epsilon^2)\right),
\]
which implies
\[
|y+z| - |z| = \frac{y\cdot z}{|z|} + \frac{|y|^2}{2|z|} + O\left( \frac{|y|^4}{|z|^3} \right).
\]
Therefore,
\[
|y+z| - |z| - \frac{y\cdot z}{|z|} = \frac{|y|^2}{2|z|} + O\left( \frac{|y|^3}{|z|^2} \right),
\]
and finally,
\[
\left||y+z| - |z| - \frac{y\cdot z}{|z|}\right| \lesssim \frac{|y|^2}{|z|},
\]
and so, note that
\[
|y+z| = |z| + \frac{y \cdot z}{|z|} + R(y,z),
\]
where the remainder \( R(y,z) \) satisfies
\[
|R(y,z)| \lesssim \frac{|y|^2}{|z|}
\quad \text{when} \quad |y| \ll |z|.
\]
Thus,
\[
|y+z| = |z| + \frac{y \cdot z}{|z|} + O\left( \frac{|y|^2}{|z|} \right),
\]
and consequently,
\[
-|y+z| = -|z| - \frac{y\cdot z}{|z|} + O\left( \frac{|y|^2}{|z|} \right).
\]
Applying the exponential function, we obtain
\[
e^{-|y+z|} = e^{-|z|-\frac{y \cdot z}{|z|}} e^{O\left( \frac{|y|^2}{|z|} \right)}.
\]
Using the Taylor expansion for small \(\eta\), 
\[
e^{\eta} = 1 + \eta + O(\eta^2),
\quad \text{thus} \quad |e^{\eta} - 1| \lesssim |\eta|,
\]
we deduce that
\[
e^{O\left( \frac{|y|^2}{|z|} \right)} = 1 + O\left( \frac{|y|^2}{|z|} \right).
\]
Therefore,
\[
e^{-|y+z|} = e^{-|z|-\frac{y \cdot z}{|z|}} \left( 1 + O\left( \frac{|y|^2}{|z|} \right) \right),
\]
and
\[
e^{-|y+z|} - e^{-|z|-\frac{y \cdot z}{|z|}} = e^{-|z|-\frac{y \cdot z}{|z|}} O\left( \frac{|y|^2}{|z|} \right).
\]
Thus,
\[
\left| e^{-|y+z|} - e^{-|z|-\frac{y \cdot z}{|z|}} \right| \lesssim |z|^{-1} e^{-|z|} e^{-\frac{y \cdot z}{|z|}} |y|^2.
\]
Now, since \( y \cdot z / |z| \) is of size comparable to \( |y| \) when \( |y| \ll |z| \), we have
\[
e^{-\frac{y \cdot z}{|z|}} \lesssim e^{|y|}.
\]
Hence,
\[
\left| e^{-|y+z|} - e^{-|z|-\frac{y \cdot z}{|z|}} \right| \lesssim |z|^{-1} e^{-|z|} (1 + |y|^2) e^{|y|},
\]
$$
\left| e^{-|y+z|}- e^{-|z|-\frac{y \cdot z}{|z|}}\right| \lesssim |z|^{-1} e^{-|z|}\left(1+|y|^{2}\right) e^{|y|} .
$$
Inserted into \eqref{2}, this yields
$$
\left|\int_{|y|<\frac{3}{4}|z|} \Phi_{\omega_1}^{3}(y)\left[\Phi_{\omega_2}(y+z)-\kappa e^{\sqrt{1-\omega_{2}^2}(-|z|-\frac{y \cdot z}{|z|})}\right] \mathrm{d} y\right| \lesssim|z|^{-1} e^{-\sqrt{\omega_{\star}}|z|} .
$$
Next, using \eqref{ground}, we observe
$$
\int_{|y|>\frac{3}{4}|z|} \Phi_{\omega_1}^{3}(y) \Phi_{\omega_2}(y+z) \mathrm{d} y \lesssim e^{-3 \sqrt{\omega_{\star}}|z|},
$$
and
$$
\int_{|y|>\frac{3}{4}|z|} \Phi_{\omega_1}^{3}(y) e^{\sqrt{1-\omega_2^2}(-|z|-\frac{y \cdot z}{|z|})} \mathrm{d} y \lesssim e^{-\sqrt{\omega_{\star}}|z|} \int_{|y|>\frac{3}{4}|z|} \Phi_{\omega_1}^{3}(y) e^{\sqrt{\omega_{\star}}|y|} \mathrm{d} y \lesssim e^{-3 \sqrt{\omega_{\star}}|z|} .
$$
Gathering these estimates, we have proved
$$
\left|\int_{\R} \Phi_{\omega_1}^{3}(y) \Phi_{\omega_2}(y+z) \mathrm{d} y-\kappa e^{-\sqrt{1-\omega_{2}^2}|z|} \int_{\R} \Phi_{\omega_1}^{3}(y) e^{\sqrt{1-\omega_{2}^2}(-\frac{y \cdot z}{|z|})} \mathrm{d} y \right|\lesssim \left| z\right|^{-1} e^{-\sqrt{\omega_{\star}}|z|}.
$$
Hence, the identity $\int_{\R} \Phi_{\omega_1}^{3}(y) e^{-\frac{y \cdot z}{| z |}} \mathrm{~d} y=\int_{\R} \Phi_{\omega_1}^{3}(y) e^{-y} \mathrm{~d} y$ (recall that $\Phi$ is radially symmetric) and the definition of $g_{0}$ imply \eqref{1}.

(iii)   First, using the Taylor  expansion, it holds
$$
\left|G-3\left| R_{1}^{(1)}\right|^{2} R_{2}^{(1)}\right|\lesssim |R_{2}^{(1)}|^{2}+\left|R_{2}^{(1)}\right|^{2}\left|R_{1}^{(1)}\right| .
$$
where $G=f(R_{1}^{(1)}+R_2^{(1)})-f(R_1^{(1)})-f(R_2^{(1)}).$  
Thus, using $(i)$, we obtain 
\begin{equation}\label{3}
\left|\left\langle G, \partial_{x} R_{1}^{(1)}\right\rangle- H(z)\right| \lesssim \int_{\R} |\Phi_{\omega_1}(y) \Phi_{\omega_2}(y+z)|^2 \mathrm{d} y \lesssim e^{-6\sqrt{\omega_{\star}}|z|},
\end{equation}
where we set $$H(z)=\int_{\R} \partial_{x}\Phi_{\omega_1}^{3}(y) \Phi_{\omega_2}(y+z) \mathrm{d} y.$$ Second, we claim that there exists a function $g:[0, \infty) \rightarrow \mathbb{R}$ such that $H(z)=c_{1} \frac{z}{|z|} g(|z|)$. 
Indeed, remark by using $z+r>0$
$$
\begin{aligned}
H\left(z\right) & =3 \int_{\R} \frac{y}{|y|} q_{\omega_1}^{\prime}(|y|) q_{\omega_1}^{2}(|y|) q_{\omega_2}\left(\left|y+z\right|\right) \mathrm{d} y \\
& =3 \int_{\R} \frac{y}{|y|} q_{\omega_1}^{\prime}(|y|) q_{\omega_1}^{2}(|y|) q_{\omega_2}\left(\left|y+z\right|\right) \mathrm{d} y.
\end{aligned}
$$
Thus, we set
$$
g(r)=\frac{ H\left(z\right)}{c_{1}} \quad \text { so that } \quad H\left(z\right)=c_{1} g(r).
$$
Consider  $y=x\frac{z}{|z|}$, then
$$
\begin{aligned}
H(z) & =3 \int_{\R} \frac{y}{|y|} q_{\omega_1}^{\prime}(|y|) q_{\omega_1}^{2}(|y|) q_{\omega_2}(|y+z|) \mathrm{d} y \\
& =3\frac{z}{|z|} \int_{\R} \frac{x}{|x|} q_{\omega_1}^{\prime}(|x|) q_{\omega_1}^{2}(|x|) q_{\omega_2}\left(\left|x+|z| \right|\right) \mathrm{d} x\\
&= H\left(|z| \right)=c_{1} \frac{z}{|z|} g(z)
\end{aligned}
$$
Together with \eqref{3}, this proves \eqref{4}. The proof of \eqref{5} is the same. Finally, proceeding as in the proof of \eqref{1}, we have that $$\left|H(r)-\kappa e^{-\sqrt{1-\omega_2^2}r}\int_{\R}\partial_{y}\Phi_{\omega_1}(y)e^{-y}dy\right|\lesssim r^{-1}e^{-\sqrt{\omega_{\star}}r}.$$ Therefore, applying integration by parts we follow \eqref{3.5}.
 
\end{proof}
%%%%%%%%%%%%%%%%%%%%%%%%%%%%%%%%%%%%%%%%%
Notice that the solutions $\vec u(t)$ of \eqref{system2} with the initial data $\vec u_0=(u_0,\rho_0,v_0,n_0)^T$ can be equivalently found from the following integral equality
   \begin{equation}\label{integral-form}
       \vec u(t)=\begin{pmatrix}
           u(t)\\
           \rho(t)\\
           v(t)\\n(t)
       \end{pmatrix}
       =\bm{G}(t)\vec u_0+\int_0^t\bm{G}(t-t')\vec F(\vec u(t'))\dd t',
   \end{equation}
   where $\bm{G}={\rm diagonal}(\bm{G}_1(t),\bm{G}_2(t))$,
   \[  
   \widehat{\bm{G}}_1(t)=
   \begin{pmatrix}
    \cos(t\sqrt{1+\xi^2})&-\frac{1}{ \sqrt{1+\xi^2}}\sin(t\sqrt{1+\xi^2})\\
    \sqrt{1+\xi^2}\sin(t\sqrt{1+\xi^2})&\cos(t\sqrt{1+\xi^2})
   \end{pmatrix},
   \]
   \[
   \widehat{\bm{G}}_2(t)=
   \begin{pmatrix}
    \cos(t\xi)&\ii\sin(t\xi)\\
    \ii\sin(t\xi)&\cos(t\xi)
   \end{pmatrix},
   \]
   and
  $\vec F(\vec u)=(0,\al uv+\beta|u|^2u,0,(|u|^2)_x)^T$.
  \begin{theorem}\label{wellpo}
   Let $s>1/2$, $r\in\rr$ and $Y^s=H^s(\rr)\,(\text{or}\,\dot{H}^s(\rr))$.  For any initial data $\vec u_0\in X^{s,r} :=Y^{s} \times Y^{s-1} \times Y^{r} \times Y^{r}  $, there exists a time existence $T>0$ such that \eqref{system2} is well-posed in $X^{s,r}$ provided that   $(r,s)\in \sett{r< s, \;r+1/2\leq2s}\cup  \sett{r< s, \;r+1/2\leq2s}$ and
   $(r,s)\in \sett{s\leq r+1, \;r>-1/2}\cup  \sett{s< r+1, \;r\geq-1/2}$. 
      Moreover, $Q_1$, $Q_2$, and the energy functional  $E$ are time-invariant, where
     \begin{equation}\label{ener1}
E( \vec u)=\int_{\R}\left(|u|^{2}+\left|\rho^{2}\right|+\left|u_{x}\right|^{2}+\alpha|u|^{2} v+\frac{\beta}{2}|u|^{4}+\frac{\alpha}{2} v^{2}+\frac{\alpha}{2} n^{2}\right) \, d x \end{equation}
\begin{equation}\label{mass1}
Q_1(\Vec{u})=2 \Re \int_{\R} u_{x} \bar{\rho}\,dx-\alpha \int_{\R} n v \,d x
\end{equation}
and
\begin{equation}\label{mass2}
Q_2(\Vec{u})=2 \Im \int_{\R} \bar{u} \rho \,d x.
\end{equation}
for the solution $\vec u=(u,\rho,v,n)^T$.
  \end{theorem}
  \begin{proof}
According the unitary group of $\bm{G}$, the proof is derived by combining the results given in \cite[Theorem 1]{hss2013} and
 \cite[Theorem 1.1]{Nakoz}.
\end{proof}
 % \begin{remark}
 %    The above theorem allows us to infer that system \eqref{system2} is well-posed in $X$.
 %\end{remark} 

\section{Coercivity property}\label{section-coer}
In this section, we aim to establish key results essential for the development of our findings. Specifically, we will begin by introducing certain quantities conserved by the system's flow and subsequently construct a coercive property associated with our system, allowing us to control future terms to be analyzed.

Note that \eqref{system2} can be written as:
\begin{equation}\label{hamilt}
\frac{d \Vec{u}}{d t}=J E^{\prime}(\Vec{u}) ,
\end{equation}
where $\Vec{u}=(u, \rho, v, n)$, $J$ is a skew-symmetrically linear operator defined by
$$
J=\left(\begin{array}{cccc}
0 & -\frac{1}{2} & 0 & 0 \\
\frac{1}{2} & 0 & 0 & 0 \\
0 & 0 & 0 & \frac{1}{\alpha} \partial_x \\
0 & 0 & \frac{1}{\alpha} \partial_x & 0
\end{array}\right)
$$

and

\begin{equation}\label{energy1}
 E^{\prime}(\Vec{u})=\left(\begin{array}{c}
-2 u_{x x}+2 u+2 \alpha u v+2 \beta|u|^{2} u  \\
2 \rho \\
\alpha|u|^{2}+\alpha v \\
\alpha n
\end{array}\right) . 
\end{equation}

Differentiating \eqref{energy1} with respect to $\vec u$, we have
\begin{equation}\label{energy2}
 E^{\prime \prime}(\Vec{u}) \vec{\eta}=\left(\begin{array}{c}
\left(-\partial_{x}^{2}+2+2 \alpha v+2 \beta|u|^{2}\right) \eta_{1}+4 \beta u R e\left(u \overline{\eta_{1}}\right)+2 \alpha u \eta_{3}  \\
2 \eta_{2} \\
2 \alpha \Re\left(u \overline{\eta_{1}}\right)+\alpha \eta_{3} \\
\alpha \eta_{4}
\end{array}\right)   
\end{equation}
where   $\vec{\eta}=\left(\eta_{1}, \eta_{2}, \eta_{3}, \eta_{4}\right)$.
Now, following what was stated in the previous theorem, we have these two quantities $R_{1}^{(1)}$ and $Q_{2}$ conserved by the flow, moment and mass respectively, 

\begin{equation*}\label{mom1}
Q_1(\Vec{u})=2 \Re \int_{\R} u_{x} \bar{\rho}\,dx-\alpha \int_{\R} n v \,d x
\end{equation*}
and
\begin{equation*}\label{mom2}
Q_2(\Vec{u})=2 \Im \int_{\R} \bar{u} \rho \,d x.
\end{equation*}
Hence, for any $t \in \R$, we have $\Vec{u}(t)$ is a flow of \eqref{system2},
\begin{equation*}\label{mom1and2}
Q_1(\Vec{u}(t))=Q_1(\Vec{u}(0)) \, \, \text { and } \, \,  Q_2(\Vec{u}(t))=Q_2(\Vec{u}(0)).
\end{equation*}
Differentiating \eqref{mom1and2} with respect to $\Vec{u}$, respectively, we have
$$
Q_1^{\prime}(\Vec{u})=\left(\begin{array}{c}
-2 \rho_{x} \\
2 u_{x} \\
-\alpha n \\
-\alpha v
\end{array}\right), \quad Q_1^{\prime \prime}(\Vec{u})=\left(\begin{array}{cccc}
0 & -2 \partial_x & 0 & 0 \\
2 \partial_x & 0 & 0 & 0 \\
0 & 0 & 0 & -\alpha \\
0 & 0 & -\alpha & 0
\end{array}\right)
$$
and
$$
Q_2^{\prime}(\Vec{u})=\left(\begin{array}{c}
-2 i \rho  \\
2 i u \\
0 \\
0
\end{array}\right), \quad Q_2^{\prime \prime}(\Vec{u})=\left(\begin{array}{cccc}
0 & -2 i & 0 & 0 \\
2 i & 0 & 0 & 0 \\
0 & 0 & 0 & 0 \\
0 & 0 & 0 & 0
\end{array}\right).
$$

			Now, we define some operators that will be used later on. Indeed,
			let us $\mathcal{S}_{j}:  X \to \R$ the functional defined by
			\begin{equation*}\label{defSj}
				\mathcal{S}_{j}(\Vec{u}):=E(\Vec{u})-\omega_j Q_2(\Vec{u}), \, \, \quad \text{$j=1,2$}.
			\end{equation*}

			Note that by the Sobolev embedding, these functionals are well-defined.

    We consider $\Vec{\Phi}_j=(\Phi_{j}^{(1)}, \Phi_{j}^{(2)}, \Phi_{j}^{(3)}, \Phi_{j}^{(4)})$ as in \eqref{standing}, %revisar un poco la escritura aqui, por que depronto podemos mejorar como expresar la idea
then it follows from \eqref{system3} that
\begin{equation*}\label{pointcrit}
  \mathcal{S}_j^{\prime}(\Vec{\Phi}_j)=\left(\begin{array}{c}
2 \Phi_{j}^{(1)}-2 \partial_{xx}\Phi_{j}^{(1)}+2 \alpha \Phi_{j}^{(1)} \Phi_{j}^{(3)}+2 \beta|\Phi_{j}^{(1)}|^{2} \Phi_{j}^{(1)}+2 i \omega_j \Phi_{j}^{(2)} \\
2\Phi_{j}^{(2)}-2 i \omega_j \Phi_{j}^{(1)} \\
\alpha|\Phi_{j}^{(1)}|^{2}+\alpha \Phi_{j}^{(3)} \\
\alpha \Phi_{j}^{(4)}
\end{array}\right)=\left(\begin{array}{c} 0 \\0 \\0 \\0 \end{array}\right).  
\end{equation*}
Now we define, for all $j \in \{1,2\}$, an operator $H_{j}$ from $X$ to $X^{\star}$ by
\begin{equation*}\label{defHj}
H_{j}=\mathcal{S}_j^{\prime \prime}(\Vec{\Phi}_{j}), 
\end{equation*}
 that is,
$$
\begin{aligned}
H_{j} \vec{\eta}&=  \left(\begin{array}{c}
2\left(-\partial_{x}^{2}+1+\alpha \Phi_{j}^{(3)}+\beta|\Phi_{j}^{(1)}|^{2}\right) \eta_{1}+4 \beta \Phi_{j}^{(1)} \Re\left(\Phi_{j}^{(1)} \overline{\eta_{1}}\right)+2 \alpha \Phi_{j}^{(1)} \eta_{3}+2 i \omega_j \eta_{2} \\
2\left(\eta_{2}-i \omega_j \eta_{1}\right) \\
2 \alpha \Re\left(\Phi_{j}^{(1)}\overline{\eta_{1}}\right)+\alpha \eta_{3} \\
\alpha \eta_{4}
\end{array}\right), 
\end{aligned}
$$
where $\vec{\eta}=\left(\eta_{1}, \eta_{2}, \eta_{3}, \eta_{4}\right)$. 

Then based on Lemma 2 of \cite{yin}, we can establish the following result. 
 
\begin{lemma}\label{kernel}
We have for system \eqref{system1} that, for all $j \in \{1,2\}$, 
    $\mathcal{S}_j^{\prime}(\vec{\Phi}_{j})=0$  
    and $$Ker(H_j)=\text{span}\{\partial_{x}\vec{\Phi}_{j},\vec{\Upsilon}_{{j}}\},
    $$
    where $\vec{\Upsilon}_{j}=(i \Phi_{j}^{(1)}, -\omega_j \Phi_{j}^{(1)},0,0) $  and $\mathcal{S}_j, H_j$ are defined by \eqref{defSj} and \eqref{defHj}, respectively.
\end{lemma} 
\begin{lemma}\label{lemma2}
 For $j=1,2$, let $\vec{\Gamma}_{j}=\left(\frac{1}{2}\partial_{\omega} \Phi_{j}^{(1)}, i \frac{1}{2}\omega_j \partial_{\omega} \Phi_{j}^{(1)},-\frac{1}{2} \Phi_{j}^{(1)} \partial_{\omega} \Phi_{j}^{(1)}, 0\right)$. We have
\begin{equation*}
H_j \vec{\Gamma}_{j}=\vec{\Psi}_{j}, 
\end{equation*}
where $\vec{\Psi}_{j}=(2\omega_j \Phi_{j}^{(1)},0,0,0)$.
Then, if $\omega_j \neq 0$, we have

\begin{equation*}
\left\langle H_j \vec{\Gamma}_{j}, \vec{\Gamma}_{j}\right\rangle<0. 
\end{equation*}
\end{lemma} 

\begin{proof} Fixed $j \in \{1,2\}$. We know from \eqref{system4} that 
$$
-\partial_{x}^{2} \Phi_{j}^{(1)}+\left(1-\omega_j^{2}\right) \Phi_{j}^{(1)}-\alpha(\Phi_{j}^{(1)})^{3}+ \beta(\Phi_{j}^{(1)})^{3}=0,
$$
then we get

\begin{equation*}
-\partial_{x}^{2} \partial_{\omega} \Phi_{j}^{(1)}+\left(1-\omega_j^{2}\right) \partial_{\omega} \Phi_{j}^{(1)}-3 \alpha(\Phi_{j}^{(1)})^{2} \partial_{\omega} \Phi_{j}^{(1)}+3 \beta(\Phi_{j}^{(1)})^{2} \partial_{\omega} \Phi_{j}^{(1)}=2 \omega_j \Phi_{j}^{(1)}. 
\end{equation*}

Thus,
$$
\begin{aligned}
 H_j \vec{\Gamma}_{j} &=\left(\begin{array}{c}
\left(-\partial_{x}^{2}+1+\alpha \Phi_{j}^{(3)}+\beta|\Phi_{j}^{(1)}|^{2}\right)\partial_{\omega} \Phi_{j}^{(1)} +4 \beta \Phi_{j}^{(1)} \Re\left(\Phi_{j}^{(1)} \overline{\partial_{\omega} \Phi_{j}^{(1)}}\right)...\\...+2 \alpha \Phi_{j}^{(1)}(-\frac{1}{2}\Phi_{j}^{(1)}\partial_{\omega} \Phi_{j}^{(1)})-2  \omega_j^2 \partial_{\omega} \Phi_{j}^{(1)}\\
i \omega_j \partial_{\omega} \Phi_{j}^{(1)}-i \omega_j \partial_{\omega} \Phi_{j}^{(1)} \\
-\Phi_{j}^{(1)} \partial_{\omega} \Phi_{j}^{(1)}+\Phi_{j}^{(1)} \partial_{\omega} \Phi_{j}^{(1)} \\
0
\end{array}\right)\\
&=\left(\begin{array}{c}
2 \omega_j \Phi_{j}^{(1)} \\
0 \\
0 \\
0
\end{array}\right) .
\end{aligned}
$$

Hence, we have 

\begin{equation*}
\left\langle H_j \vec{\Gamma}_{j}, \vec{\Gamma}_{j}\right\rangle<0,
\end{equation*}
for $\omega_j \neq 0$.
\end{proof}
\begin{lemma}\label{lemmacoercivity}
For $j=1,2$ assume $\omega_j \neq 0$, for any given $\vec{\xi}=(\xi, \eta, \zeta, \iota)\in X$ satisfying

\begin{equation*}
\left\langle\vec{\xi}, \vec{\Upsilon}_{j}\right\rangle=\left\langle\vec{\xi}, \partial_{x} \vec{\Phi}_{j}\right\rangle=\left\langle\vec{\xi}, \vec{\Psi}_{j}\right\rangle=0
\end{equation*}

there exists a positive constant $C$ such that

\begin{equation*}
\left\langle H_j \vec{\xi}, \vec{\xi}\right\rangle \geq C\|\vec{\xi}\|_{X}^{2}.
\end{equation*}
Moreover, it follows from the above that there exists $C>0$ such that
\[\begin{aligned}
					\langle H_j \vec{\xi},\vec{\xi}\rangle &\geq C\|\vec{\xi}\|_{X}^2-\langle\vec{\xi},\vec{Y}_{\omega_j}\rangle_{L^2(\R)}^2-\langle\vec{\xi},\partial_x \vec{\Phi}_{j}\rangle_{L^2(\R)}^2,
				\end{aligned}\]
            where $Y_{\omega_j}$ is the only negative eigenvalue associated with $H_j$.
        
\end{lemma}
\begin{proof}
    Similar argument can be found in (\cite{yin}, Lemma 4).
\end{proof}
Let $\varepsilon_0^j>0$, such that $$H_j \vec{Y}_{\omega_j}=-\varepsilon_0^j \vec{Y}_{\omega_j}.$$
Then, we define $$\vec{Y}_{\omega_j}^{\pm}=(Y_{\omega_k}^{(1)},\pm \varepsilon_0^jY_{\omega_j}^{(1)},Y_{\omega_j}^{(3)} ,Y_{\omega_j}^{(4)})$$ and $$\vec{Z}_{\omega_j}^{\pm}=(\pm \varepsilon_0^j Y_{\omega_j}^{(1)},Y_{\omega_j}^{(1)},Y_{\omega_j}^{(3)},Y_{\omega_j}^{(4)})$$ 
It is worth noting that, from \cite{WEIN}, the functions $\vec{Z}_{\omega_j}^{\pm}$ satisfy $\|\vec{Z}_{\omega_j}^{\pm}\|_{L^2(\R)}^2=1$ and decay to zero at infinity. Moreover, for $\epsilon>0$, \[|\vec{Z}_{\omega_j}^{\pm}|+|\partial_x \vec{Z}_{\omega_j}^{\pm}|\leq e^{-\epsilon |x|}.\]
In the same way we define,
$$Y_{k}=Y_{\omega_k}(\cdot-z_k), \,\, \, \, \, \, \vec{Y}_{k}^{\pm}=\vec{Y}_{\omega_k}^{\pm}(\cdot-z_k),\,\, \,  \, \, \, \vec{Z}_{k}^{\pm}=\vec{Z}_{\omega_k}^{\pm}(\cdot-z_k),$$
and 
$$\lambda^{\pm}=\langle \vec{\varepsilon},\vec{Z}^{\pm}\rangle, \,\, \, \, \, \, \lambda_{k}^{\pm}= \langle \vec{\varepsilon},\vec{Z}_{k}^{\pm}\rangle, \, \, with \, \, 
 \vec{Z}^{\pm}=\sum_{k=1}^{2}\vec{Z}_{k}^{\pm}.$$
   Next, we will establish a result that is a consequence of the theory established by Pego-Weinstein in \cite{WEIN}.
   \begin{lemma}\label{coer}
      The following properties hold:

(i) $\vec{Z}_j^{ \pm}$ are  eigenfunctions of $H_{j}.$

(ii)  For all  $\eta_{0}>0$, $x \in \mathbb{R}$ and  $m=1,2$,
\[
\left|Y_{j}^{ \pm,m}(x)\right|+\left|\partial_{x} Y_{j}^{ \pm,m}(x)\right|+\left|Z_{j}^{ \pm,m}(x)\right|+\left|\partial_{x} Z_{j}^{ \pm,m}(x)\right| \leq C e^{-\eta_{0} \sqrt{c}|x|}.
\]

(iii) 
$\left(\vec{Y}_{j}^{\pm}, \vec{Z}_{j}^{\pm}\right)_{L^{2}(\R)\times L^{2}(\R)}=\pm 2\varepsilon_0^j$ and $\left(\vec{Z}_{j}^{\pm}, \partial_x\vec{\Phi}_{j}\right)_{L^{2}(\R)\times L^{2}(\R)}=0.$

(iv)   \[\left(\vec{Y}_{j}^{+}, \vec{Z}_{j}^{-}\right)_{L^{2}(\R)\times L^{2}(\R)}=\left(\vec{Y}_{j}^{-}, \vec{Z}_{j}^{+}\right)_{L^{2}(\R)\times L^{2}(\R)}=0.\]
(v) There exists $C>0$ such that
\[\begin{aligned}
					\langle H_j \vec{\xi},\vec{\xi}\rangle &\geq C\|\vec{\xi}\|_{X}^2-\langle\vec{\xi},\vec{Z}_j^{+}\rangle_{L^2(\R)}^2-\langle\vec{\xi},\vec{Z}_j^{-}\rangle_{L^2(\R)}^2-\langle\vec{\xi},\partial_x \vec{\Phi}_{j}\rangle_{L^2(\R)}^2.
				\end{aligned}\]
            where $Y_j$ is the only negative eigenvalue associated with $H_j$.
   \end{lemma}
   Once the new coercivity condition is established, we proceed to state the main theorem under these conditions.
			
  \vspace{2mm}
We will now establish a result of parameter modulation that allows us to obtain some new estimates. 
\\
Let us now consider, for $k=1,2$, the following set
$$U_{\varepsilon}\left(\vec{\Phi}_{\omega_k}\right)=\left\{\vec{u} \in X: \inf_{|\xi_k^1-\xi_k^2|>\frac{2}{\sqrt{\omega_{\star}}}\log(\epsilon), \theta_k \in \R}\left\|\vec{u}-T(\theta_k)\vec{\Phi}_{\omega_k}(\cdot-z_k)\right\|_{X}<\epsilon\right\}.$$	
\begin{lemma}\label{2lemamodulation}
				There exists $\varepsilon_{0}>0$ such that for any given $\varepsilon \in\left(0, \varepsilon_{0}\right)$, if $\vec{u} \in U_{\varepsilon}\left(\vec{\Phi}_{\omega}\right)$, then there exist $\mathrm{C}^{1}$-functions
$$
\theta(t):\left[0, t^{\star}\right] \rightarrow \mathbb{R},  \quad z(t):\left[0, t^{\star}\right] \rightarrow \mathbb{R},
$$
such that if we define $\vec{\varepsilon}=(\xi, \eta, \zeta, \iota)$ satisfying

\begin{equation*}\label{ident1}
\vec{\varepsilon}(t)=\vec{u}(t)-T(\theta(t)) \vec{\Phi}_{\omega}(\cdot-z(t)):=\vec{u}(t)-\vec{R}(t).
\end{equation*}
Then for any given $\mathrm{t} \in\left[0, \mathrm{t}^{\star}\right], \vec{\varepsilon}$ satisfies the following orthogonality conditions,

\begin{equation}\label{2orthogonality}
\left\langle\vec{\varepsilon}, T(\theta(t)) \vec{\Upsilon}_{ \omega}\right\rangle=\left\langle\vec{\varepsilon}, T(\theta(t)) \partial_{x} \vec{\Phi}_{ \omega}\right\rangle=\left\langle\vec{\varepsilon}, T(\theta(t)) \vec{\Psi}_{ \omega}\right\rangle=0 
\end{equation}
where
$$\begin{aligned}
T(\theta(t)) \vec{\Upsilon}_{ \omega} & =\left(i e^{i \theta(t)} \Phi_{ \omega}^{(1)}(\cdot-z(t)),- \omega e^{i \theta(t)} \Phi_{ \omega}^{(1)}(\cdot-z(t)), 0,0\right) \\
T(\theta(t)) \partial_{x} \vec{\Phi}_{ \omega} & =\left(e^{i \theta(t)} \partial_{x} \Phi_{ \omega}^{(1)}(\cdot-z(t)), i  \omega e^{i \theta(t)} \partial_{x} \Phi_{ \omega}^{(1)}(\cdot-z(t)),-2 \Phi_{ \omega}^{(1)} \partial_{x} \Phi_{ \omega}^{(1)}(\cdot-z(t)), 0\right) \\
T(\theta(t)) \vec{\Psi}_{ \omega} & =\left(2  \omega e^{i \theta(t)} \Phi_{ \omega}^{(1)}(\cdot-z(t)), 0,0,0\right).
\end{aligned}$$

Here $\vec{R}=T(\theta(t)) \vec{\Phi}_{ \omega}(\cdot-z(t))$ is called the modulated soliton. Moreover, we have estimates

\begin{equation}\label{estimative1}
\|\vec{\varepsilon}\|_{X}+e^{-2|z|} \lesssim \varepsilon 
\end{equation}
 and
\begin{equation}\label{estimative2}
|\dot{\theta}(t)- \omega|+|\dot{z}|=O\left(\|\vec{\varepsilon}\|_{X}^{2}\right).
\end{equation}
Moreover, \begin{equation}\label{estimative2.5}
 \left|\dot{\theta}_{k} \mp \frac{z}{|z|}g(|z|)\right|\lesssim \|\vec{\varepsilon}\|_{X}^2 +e^{-4\sqrt{\omega_{\star}}|z|}+\|\vec{\varepsilon}\|_{X}O(|\dot{z}_{1}|).   
\end{equation}
			\end{lemma}
\begin{proof}
We first show the existence of $\theta$ and $y$ by the Implicit Function Theorem. Denoting
$$F_{1}(\theta, y ; \vec{u})=\left\langle\vec{\varepsilon}, T(\theta) \vec{\Upsilon}_{\omega}\right\rangle, \quad F_{2}(\theta, y; \vec{u})=\left\langle\vec{\varepsilon}, T(\theta) \partial_{x} \vec{\Phi}_{\omega}\right\rangle,$$
we have
$$
\begin{aligned}
F_{1}(\theta, y ; \vec{u})= & \left\langle u-e^{i \theta} \Phi_{\omega}^{(1)}(-y), i e^{i \theta} \Phi_{\omega}^{(1)}(-y)\right\rangle+\left\langle \rho-i  \omega e^{i \theta} \Phi_{\omega}^{(1)}(-y),- \omega e^{i \theta} \Phi_{\omega}^{(1)}(-y)\right\rangle, \\
F_{2}(\theta, y, \lambda ; \vec{u})= & \left\langle u-e^{i \theta} \Phi_{\omega}^{(1)}(-y), e^{i \theta} \partial_{x} \Phi_{\omega}^{(1)}(-y)\right\rangle+\left\langle v-i  \omega e^{i \theta} \Phi_{\omega}^{(1)}(-y), i  \omega e^{i \theta} \partial_{x} \Phi_{\omega}^{(1)}(-y)\right\rangle \\
& \left.+\left.\langle n+| \Phi_{\omega}^{(1)}(-y)\right|^{2},-2 \Phi_{\omega}^{(1)} \partial_{x} \Phi_{\omega}^{(1)}(-y)\right\rangle.
\end{aligned}
$$
Then
$$
\left.\frac{\partial\left(F_{1}, F_{2}\right)}{\partial(\theta, y)}\right|_{\left(0,0 ; \vec{\Phi}_{\omega)}\right.}=\left(\begin{array}{ccc}
a_{11} & 0  \\
0 & a_{22}  \\
\end{array}\right)
$$
where
$$
\begin{aligned}
& a_{11}=\left.\partial_{\theta} F_{1}\right|_{(0,0 ; \vec{\Phi}_{\omega})}=-\left(1+\omega^{2}\right)\left\|\Phi_{\omega}^{(1)}\right\|^{2} \\
& a_{22}=\left.\partial_{y} F_{2}\right|_{\left(0,0 ; \vec{\Phi}_{\omega}\right)}=\left(1+\omega^{2}\right)\left\|\partial_{x} \Phi_{\omega}^{(1)}\right\|^{2}+\left\|\partial_{x}\left|\Phi_{\omega}^{(1)}\right|^{2}\right\|^{2}.
\end{aligned}
$$

This means that the Jacobian matrix of the derivative of $(\theta, y) \rightarrow\left(F_{1}, F_{2}\right)$ is nondegenerate at $(\theta, y)=(0,0)$. Hence, by Implicit Function Theorem, we obtain the existence of $(\theta, y)$ in a neighborhood of $(0,0)$ satisfying
$$
\left(F_{1}, F_{2}\right)(\theta(t),z(t) ; \vec{u}(t))=0 .
$$

To verify the modulation parameters $(\theta, y)$ are $C^{1}$, we can use the equation of $\vec{\varepsilon}$ and regularization arguments, similar to the method in Martel-Merle \cite{marmer-1}. 
We now want to verify \eqref{estimative2}. From  \eqref{hamilt} and \eqref{ident1}, we have

\begin{equation*}
\partial_{t} \vec{\varepsilon}+i \dot{\theta} \vec{R}-\dot{z} \partial_{x} \vec{R} =J E^{\prime}(\vec{R}+\vec{\varepsilon})
\end{equation*}
Moreover, from Taylor's expansion, we have

\begin{equation*}
E^{\prime}(\vec{R}+\vec{\varepsilon})=E^{\prime}(\vec{R})+E^{\prime \prime}(\vec{R}) \vec{\varepsilon}+O\left(\|\vec{\varepsilon}\|_{X}^{2}\right) 
\end{equation*}

Noting that, the fact that $\Phi_{\omega}^{(1)}$ is a solution of

\begin{equation*}
-\partial_{x}^{2} \Phi_{\omega}^{(1)}+\left(1- \omega^{2}\right) \Phi_{\omega}^{(1)}-\alpha(\Phi_{\omega}^{3})+\beta(\Phi_{\omega}^{3})=0, 
\end{equation*}
implies 
\begin{equation}\label{estimative3}
\partial_{t} \vec{\varepsilon}+i(\dot{\theta}- \omega) A \vec{R}-\dot{z} \partial_{x} \vec{R}=J E^{\prime \prime}(\vec{R}) \vec{\varepsilon}+O\left(\|\vec{\varepsilon}\|_{X}^{2}\right)
\end{equation}
where
$$
A=\left(\begin{array}{llll}
1 & 0 & 0 & 0 \\
0 & 1 & 0 & 0 \\
0 & 0 & 0 & 0 \\
0 & 0 & 0 & 0
\end{array}\right)
$$

We differentiate $\left\langle\vec{\varepsilon}, T(\theta) \vec{\Upsilon}_{\omega}\right\rangle=0$ in \eqref{2orthogonality} with respect to time $t$ to get

\begin{equation*}
\left\langle\partial_{t} \vec{\varepsilon}, T(\theta) \vec{\Upsilon}_{ \omega}\right\rangle=-\left\langle\vec{\varepsilon}, \partial_{t}\left(T(\theta) \vec{\Upsilon}_{ \omega}\right)\right\rangle.
\end{equation*}
Denoting
$$
\operatorname{Mod}(t)=(\dot{\theta}(t)- \omega, \dot{z}(t))
$$
we have

\begin{equation}\label{estimative4}
\left\langle\vec{\varepsilon}, \partial_{t}\left(T(\theta) \vec{\Upsilon}_{\omega}\right)\right\rangle=O\left((1+|\operatorname{Mod}(t)|)\|\vec{\varepsilon}\|_{X}\right)
\end{equation}
Because of the orthogonality condition and taking the inner product of \eqref{estimative3} with $T(\theta) \vec{\Upsilon}_{ \omega}$, we get
$$
\begin{aligned}
& \left\langle\partial_{t} \vec{\varepsilon}, T(\theta) \vec{\Upsilon}_{\omega}\right\rangle \\
= & -(\dot{\theta}- \omega)\left\langle i A \vec{R}, T(\theta) \vec{\Upsilon}_{ \omega}\right\rangle+\dot{z}\left\langle\partial_{x} \vec{R}, T(\theta) \vec{\Upsilon}_{ \omega}\right\rangle+O\left(\|\vec{\varepsilon}\|_{X}^{2}\right)\\
= & (\dot{\theta}- \omega)(1+\omega^{2})\|\Phi_{\omega}^{(1)}\|_{L^2(\R)}^2 +O\left(\|\vec{\varepsilon}\|_{X}^{2}\right).
\end{aligned}
$$
Hence, from \eqref{estimative4} implies that
\begin{equation*}
(1+\omega^{2})\|\Phi_{\omega}^{(1)}\|_{L^2(\R)}(\dot{\theta}- \omega)=O\left((1+|\operatorname{Mod}(t)|)\|\vec{\varepsilon}\|_{X}\right)+O\left(\|\vec{\varepsilon}\|_{X}^{2}\right).
\end{equation*}
Taking the inner product of \eqref{estimative3} with $T(\theta) \partial_{x} \vec{\Phi}_{\omega}$, respectively, by similar arguments, we get
\begin{equation*}
\left[\left(1+ \omega^{2}\right)\left\|\partial_{x} \Phi_{\omega}^{(1)}\right\|^{2}+\left\|\partial_{x}\left|\Phi_{\omega}^{(1)}\right|^{2}\right\|^{2}\right] \dot{z}=O\left((1+|\operatorname{Mod}(t)|)\|\vec{\varepsilon}\|_{X}\right)+O\left(\|\vec{\varepsilon}\|_{X}^{2}\right).
\end{equation*}
Indeed, we have
\begin{equation*}
M\operatorname{Mod}(t)=O\left((1+|\operatorname{Mod}(t)|)\|\vec{\varepsilon}\|_{X}\right)+O\left(\|\vec{\varepsilon}\|_{X}^{2}\right) 
\end{equation*}
where
$$
M=\left(\begin{array}{ccc}
\left(1+ \omega^{2}\right)\left\|\Phi_{\omega}^{(1)}\right\|^{2} & 0  \\
0 & \left(1+ \omega^{2}\right)\left\|\partial_{x} \Phi_{\omega}^{(1)}\right\|^{2}+\left\|\partial_{x}\left|\Phi_{\omega}^{(1)}\right|^{2}\right\|^{2} 
\end{array}\right)
$$
is an invertible matrix. Therefore,
\begin{equation*}
|\operatorname{Mod}(t)| \leq C\|\vec{\varepsilon}\|_{X}+O\left(\|\vec{\varepsilon}\|_{X}^{2}\right).
\end{equation*}
Finally, note that for $f(u)=|u|^2u$ and $G=f\left(R_{1}^{(1)}+R_{2}^{(1)}\right)-f\left(R_{1}^{(1)}\right)-f\left(R_{2}^{(1)}\right)$, we have
$$
\begin{aligned}
& \frac{\mathrm{d}}{\mathrm{~d} t}\left\langle\varepsilon_2, \partial_{x} R_{1}^{(1)}\right\rangle=\left\langle\partial_{t} \varepsilon_2, \partial_{x} R_{1}^{(1)}\right\rangle+\left\langle\varepsilon_2, \partial_{t}\left(\partial_{x} R_{1}^{(1)}\right)\right\rangle \\
& =\left\langle -\partial_{xx} \varepsilon_1-\varepsilon_1+f^{\prime}\left(R_{1}^{(1)}\right) \varepsilon_1, \partial_{x} R_{1}^{(1)}\right\rangle+\left\langle f(R_{1}^{(1)}+R_{2}^{(1)}+\varepsilon_1)-f(R_{1}^{(1)}+R_{2}^{(1)})-f^{\prime}(R_{1}^{(1)}+R_{2}^{(1)}) \varepsilon_1, \partial_{x} R_{1}^{(1)}\right\rangle \\
& +\left\langle\left(f^{\prime}(R_{1}^{(1)}+R_{2}^{(1)})-f^{\prime}\left(R_{1}^{(1)}\right)\right) \varepsilon_1, \partial_{x} R_{1}^{(1)}\right\rangle+\left\langle\operatorname{Mod}_{\varepsilon_1}, \partial_{x} R_{1}^{(1)}\right\rangle+\left\langle G, \partial_{x} R_{1}^{(1)}\right\rangle+\left\langle D, \partial_{x} R_{1}^{(1)}\right\rangle \\
& -\left\langle\varepsilon_2,\left(i\dot{\theta}_{1}  \partial_{x} R_{1}^{(1)}-\dot{z}_{1}\partial_{xx} R_{1}^{(1)}\right)\right\rangle,
\end{aligned}
$$
where  we use the fact that 
$$ -\partial_{xx}R_{k}^{(1)}+R_{k}^{(1)}+\alpha R_{k}^{(1)}R_{k}^{(3)}+\beta |R_{k}^{(1)}|^2 R_{k}^{(1)}=0 \, \, for \, \, k=1,2,$$ 
$$Mod_{\varepsilon_1}=-\sum_{k=1,2}i\dot{z}_k \omega_{k}\partial_{x}R_{k}^{(1)}-\sum_{k=1,2}\dot{\theta}_k \omega_{k} R_{k}^{(1)}$$
and $$\begin{aligned}
D&=\beta\left|\sum_{k=1,2}R_{k}^{(1)}\right|^2\varepsilon_1+2\beta \sum_{k=1,2}\R e(R_{k}^{(1)}\varepsilon_1)\varepsilon_1+\beta \sum_{k=1,2}R_{k}^{(1)}|\varepsilon_1|^2\\
  &\quad +2\beta \left(\sum_{k=1,2}\R e(R_{k}^{(1)}\varepsilon_1)\right)\sum_{k=1,2}R_{k}^{(1)}+\alpha \sum_{k=1,2}R_{k}^{(3)}\varepsilon_1+\alpha \sum_{k=1,2}R_{k}^{(1)}\varepsilon_3 +\beta \sum_{j\neq k=1,2}|R_{k}^{(1)}|^2R_{j}^{(1)}\\
  &\quad +\alpha\sum_{j\neq k=1,2} R_{k}^{(1)}R_{j}^{(3)}
\end{aligned} $$
Observe that since $\partial_{x} R_{1}^{(1)}$ satisfies $\partial_{xxx} R_{1}^{(1)}-\partial_{x} R_{1}^{(1)}+f^{\prime}\left(R_{1}^{(1)}\right) \partial_{x} R_{1}^{(1)}=0$, the first term is zero. For the second term, using the fact that solitons are bounded and the Taylor's expansion it clearly follows that
\begin{equation*}
\left|\left\langle f(R_{1}^{(1)}+R_{2}^{(1)}+\varepsilon_1)-f(R_{1}^{(1)}+R_{2}^{(1)})-f^{\prime}(R_{1}^{(1)}+R_{2}^{(1)}) \varepsilon, \partial_{x} R_{1}^{(1)}\right\rangle\right| \lesssim\|\varepsilon\|_{X}^{2} . 
\end{equation*}

For the third term, using Lemma \ref{solitons},
\begin{equation*}
\left|\left\langle\left(f^{\prime}(R_{1}^{(1)}+R_{2}^{(1)})-f^{\prime}\left(R_{1}^{(1)}\right)\right) \varepsilon_1, \partial_{x} R_{1}^{(1)}\right\rangle\right| \lesssim e^{-4\sqrt{\omega_{\star}}|z|}\|\vec{\varepsilon}\|_{X} . 
\end{equation*}

For the fourth term, first observe that
$$
\begin{aligned}
\left\langle\operatorname{Mod}_{\varepsilon_2}, \partial_{x} R_{1}^{(1)}\right\rangle= & -i\dot{z}_1 \omega_1 \left\|\partial_{x} R_{1}^{(1)}\right\|_{L^{2}(\R)}^{2}-i\dot{z}_2 \omega_2\left\langle\partial_{x} R_{2}^{(1)}, \partial_{x} R_{1}^{(1)}\right\rangle \\
& -\sum_{k=1,2}\dot{\theta}_k \omega_k \left\langle  R_{k}^{(1)}, \partial_{x} R_{1}^{(1)}\right\rangle .
\end{aligned}
$$

Thus from Lemma \ref{solitons} and \eqref{estimative2}
$$
\left\langle\operatorname{Mod}_{\varepsilon_2}, \partial_{x} R_{1}^{(1)}\right\rangle=-\dot{\theta}_{1}\omega_1\left \langle \partial_{x} R_{1}^{(1)},R_{1}^{(1)}\right\rangle+O\left(\left\|\varepsilon
\right\|_{X}^2+|\dot{z}_2|e^{-4\sqrt{\omega_{\star}}|z|}+|\dot{\theta}_2|e^{-4\sqrt{\omega_{\star}}|z|}\right).
$$
Now, from \eqref{4} we have
$$
\left|\frac{\left\langle G, \partial_{x} R_{1}^{(1)}\right\rangle}{\left\|\partial_{x} R_{1}^{(1)}\right\|_{L^{2}(\R)}^{2}}-\frac{z}{|z|} g(|z|)\right| \lesssim e^{-4\sqrt{\omega_{\star}}|z|}.
$$

Next using the definition of $D$ and Lemma \ref{solitons} we get,
$$
\left|\left\langle D, \partial_{x} R_{1}^{(1)}\right\rangle\right| \lesssim e^{-4\sqrt{\omega_{\star}}|z|}.
$$

For the last term, we use \eqref{2orthogonality},
$$
\left|\left\langle\varepsilon_2,\left(\dot{\theta}_{1}  \partial_{x} R_{1}^{(1)}-\dot{z}_{1}\partial_{xx} R_{1}^{(1)}\right)\right\rangle\right| \lesssim \|\vec{\varepsilon}\|_{X}O(|\dot{z}_{1}|).
$$

Gathering all these estimates,
$$
\begin{gathered}
\left|\dot{\theta}_{1}-\frac{z}{|z|} g(|z|)\right| \lesssim \|\vec{\varepsilon}\|_{X}^2 +e^{-4\sqrt{\omega_{\star}}|z|}+\|\vec{\varepsilon}\|_{X}O(|\dot{z}_{1}|).
\end{gathered}
$$
In the other cases, the procedure is similar.
\end{proof}
            Now, we will establish a standard lemma that will allow us to redefine the soliton values at sufficiently large times.
\begin{lemma}\label{linearcom}
There exists $\beta<0$ such that for any $\left(z_{1}, z_{2}\right) \in \mathbb{R}^{2}$ with $|z|$ large enough, there exist linear maps,
$$
B: \mathbb{R}^{2} \rightarrow \mathbb{R}^{2}, \quad V_{j}: \mathbb{R}^{2} \rightarrow \mathbb{R}^{2} \quad \text { for } j=1, 2.
$$
smooth in $\left(z_{1}, z_{2}, \right)$ satisfying
$$
\|B-\beta \operatorname{Id}\| \lesssim e^{-\frac{1}{2}|z|}, \quad\left\|V_{j}\right\| \lesssim e^{-\frac{1}{2}|z|} \quad \text { for } j=1, 2.
$$
and such that the function $W\left(a_{1}, a_{2}\right): \mathbb{R} \rightarrow \mathbb{R}$ defined by
$$
W\left(a_{1}, a_{2}\right)=\sum_{k=1,2}\left[B_{k}\left(a_{1}, a_{2}\right) Y_{k}+ V_{k, j}\left(a_{1}, a_{2}\right) \partial_{x} \Phi_{k}\right]
$$
satisfies for all $k=1,2$,
$$
\left\langle W\left(a_{1}, a_{2}\right), \partial_{x} R_{k}\right\rangle=0, \quad\left\langle W\left(a_{1}, a_{2}\right), Y_{k}\right\rangle=\beta a_{k}.
$$
Moreover, 
$$\left\langle\vec{W}\left(a_{1}, a_{2}\right), \vec{Z}_{k}^{-}\right\rangle=a_{k} \quad \text { and } \quad\left\langle\vec{W}\left(a_{1}, a_{2}\right), \vec{Z}_{k}^{+}\right\rangle=0,
$$ 
where
\[
\vec{W}\left(a_{1}, a_{2}\right)=\binom{W\left(a_{1}, a_{2}\right)}{-\epsilon_0^{k} W\left(a_{1}, a_{2}\right)}.
\]
\end{lemma}

\begin{proof}
Define
$$
W\left(a_{1}, a_{2}\right)=\sum_{k=1,2}\left\{B_{k} Y_{k}+ V_{k, j} \partial_{x} \Phi_{k}\right\}
$$
our goal is to solve for $B_{k}, V_{k, j}$ in function of $a_{1}, a_{2}$. Using the relations $\left\langle Y_{1}, \partial_{x} \Phi_{1}\right\rangle=$  and the estimate $\left\langle\partial_{x} \Phi_{2}, \partial_{x} \Phi_{1}\right\rangle=O\left(e^{-\frac{1}{2}|z|}\right)$, we observe that the condition $\left\langle W\left(a_{1}, a_{2}\right), \partial_{x} \Phi_{1}\right\rangle=0$ is equivalent to a linear relation between the coefficients in the definition of $W\left(a_{1}, a_{2}\right)$ of the form
$$
\left\|\partial_{x} \Phi\right\|_{L^{2}(\R)}^{2} V_{1, 1}=O\left(e^{-\frac{1}{2}|z|}\right) B_{2}+O\left(e^{-\frac{1}{2}|z|}\right) V_{2, 2}
$$

Similarly, the condition $\left\langle W\left(a_{1}, a_{2}\right), \partial_{x} \Phi_{2}\right\rangle=0$ is equivalent to
$$
\left\|\partial_{x} \Phi\right\|_{L^{2}(\R)}^{2} V_{2, j}=O\left(e^{-\frac{1}{2}|z|}\right) B_{1}+ O\left(e^{-\frac{1}{2}|z|}\right) V_{1,2}
$$

Moreover, since $\langle Y, Y\rangle=1$  and $\left\langle Y_{1}, Y_{2}\right\rangle=O\left(e^{-\frac{1}{2}|z|}\right)$, the conditions $\left\langle W\left(a_{1}, a_{2}\right), Y_{k}\right\rangle=\beta a_{k}$ write
$$
\begin{aligned}
& B_{1}=\beta a_{1}+O\left\{\left(\left|B_{2}\right|+\left|V_{2, 2}\right|\right) e^{-\frac{1}{2}|z|}\right\}, \\
& B_{2}=\beta a_{2}+O\left\{\left(\left|B_{1}\right|+\left|V_{1, 1}\right|\right) e^{-\frac{1}{2}|z|}\right\} .
\end{aligned}
$$

The existence and desired properties of $B_{k}$ and $V_{k, j}$ for $k=1,2$ and $j=1, 2$ thus follow from inverting a linear system for $|z|$ large enough. For the last part using the definition of $\beta$,
$$
\left\langle\vec{W}\left(a_{1}, a_{2}\right), \vec{Z}_{k}^{-}\right\rangle=\left\langle W\left(a_{1}, a_{2}\right),-\epsilon_{0} Y_{k}\right\rangle+\left\langle-\epsilon_{0} W\left(a_{1}, a_{2}\right), Y_{k}\right\rangle=-2 \epsilon_{0} \beta a_{k}=a_{k}
$$
and
$$
\left\langle\vec{W}\left(a_{1}, a_{2}\right), \vec{Z}_{k}^{+}\right\rangle=\left\langle W\left(a_{1}, a_{2}\right), \epsilon_{0} Y_{k}\right\rangle+\left\langle-\epsilon_{0} W\left(a_{1}, a_{2}\right), Y_{k}\right\rangle=0
$$
\end{proof} 

To show the existence of dipole  for   \eqref{system1}, we consider a sequence $T^n\to\infty$ and take $\mathbf{u}^{n}\in X$ as the solution of \eqref{system1} such that \begin{equation}\label{ident0}
\vec{u}_{n}\left(T_{n}\right)=\vec{R}_{1,n}+\vec{R}_{2,n}+\vec{W}\left(a_{1, n}, a_{2, n}\right), 
\end{equation} 
where
$$
\vec{R}_{k, n}=T(\theta_k)\vec{\Phi}_{\omega_k}\left(x-z_{k, n}\right)\, \, \, \text{for}\, \, \,  k=1,2.$$
             We consider the following lemma.
\begin{lemma}\label{direction}
				For all $t \in [T_0,T_n]$ and $j\in \{1,2\}$. The following estimate holds,
	\[\left|\frac{d}{dt}\lambda_j^{\pm}(t)\pm \epsilon_{0}^j \lambda_j^{\pm}(t)\right|\leq \|\vec{\varepsilon}\|_{X}^2+t^{-3}.\]
			\end{lemma}
\begin{proof}
    Note that 
\begin{equation*}
					\begin{aligned}
						&\frac{d}{dt}\lambda_{j}^{\pm}(t)
						=\frac{d}{dt}\langle \vec{\varepsilon},\vec{Z}_{j}^{\pm}\rangle\\
                    &=\langle \partial_{t}\vec{\varepsilon}, \vec{Z}_{j}^{\pm}\rangle
                    +\langle \vec{\varepsilon},\partial_{t}\vec{Z}_{j}^{\pm}\rangle\\&=\pm \epsilon_{0}^{j}\int_{\R}\partial_{t}\varepsilon_1 Y_{j}\,dx+\int_{\R}\partial_{t}\varepsilon_2 Y_{j}\,dx+\int_{\R}\partial_{t}\varepsilon_3 Y_{j}\,dx+\int_{\R}\partial_{t}\varepsilon_4 Y_{j}\,dx+\langle \vec{\varepsilon},\partial_{t}Z_{j}^{\pm}\rangle\\&=\pm \epsilon_{0}^{j}\int_{\R} \left(-\varepsilon_2-i\sum_{k=1}^{2}\omega_{k}R_{k}^{(1)}+\sum_{k=1}^{2}\dot{\theta}_{k}R_{k}^{(1)}+\sum_{k=1}^{2}\dot{z}_{k}\partial_{x}R_{k}^{(1)}) \right) Y_j\,dx\pm \epsilon_{0}^{j}\int_{\R} \varepsilon_1 \partial_t Y_{j}(t)\,dx\\
 &\quad+\int_{\R} \left(-\partial_{xx}\varepsilon_1+\varepsilon_1+\sum_{k=1}^{2}R_{k}^{(1)}-\sum_{k=1}^{2}\partial_{xx}R_{k}^{(1)}\right) Y_{j}^{ \pm,1}(t)\,dx \\&\quad+\alpha\int_{\R}\left( \sum_{k=1}^{2} R_{k}^{(1)} R_{k}^{(3)}+ \sum_{k=1}^{2} R_{k}^{(1)} \varepsilon_3+\sum_{k=1}^{2} R_{k}^{(3)} \varepsilon_1\right)Y_{j}\,dx\\
 &\quad +\beta \int_{\R}\left( \left(\sum_{k=1}^{2} |R_{k}^{(1)}|^2\right) \left(\sum_{k=1}^{2} R_{k}^{(1)}\right) +\sum_{k=1}^{2} |R_{k}^{(1)}|^2 \varepsilon_1\right)Y_{j}\,dx\\
 &\quad+\beta\int_{\R} \left(2\sum_{k=1}^{2} \Re(R_{k}^{(1)}\varepsilon_1)\left(\sum_{k=1}^{2} R_{k}^{(1)}\right)+2 \Re(R_{1}^{(1)}R_{2}^{(1)})\left(\sum_{k=1}^{2} R_{k}^{(1)}\right)+  \left(\sum_{k=1}^{2} R_{k}^{(1)}\right)|\varepsilon_1\right)Y_{j}\,dx\\&\quad+\int_{\R} \left(+2\sum_{k=1}^{2} \Re(R_{k}^{(1)}\varepsilon_1)\varepsilon_1+2 \Re(R_{1}^{(1)}R_{2}^{(1)})\varepsilon_1+ |\varepsilon_1|^2\varepsilon_1\right)Y_{j}\,dx\\
 &\quad-\int_{\R}\left(\sum_{k=1}^{2}\dot{\theta}_{k}\omega_{k}R_{k}^{(1)}+i\sum_{k=1}^{2}\dot{z}_{k}\omega_{k}R_{k}^{(1)}\right)Y_{j}\,dx\\
 &\quad+\int_{\R}\left(\partial_{x}\varepsilon_4+\sum_{k=1}^{2}\partial_{x}R_{k}^{(4)}- \sum_{k=1}^{2}\dot{z}_k R_{k}^{(3)}\right)Y_{j}\,dx\\&\quad+\int_{\R}\left(\partial_{x}\varepsilon_3+\sum_{k=1}^{2}\partial_{x}R_{k}^{(3)}- \sum_{k=1}^{2}\dot{z}_k R_{k}^{(4)}+\left(\left|\varepsilon_1+\sum_{k=1}^{2}R_{k}^{(1)}\right|\right)\right)Y_{j}\,dx\\&\quad-\int_{\R}\dot{z}_{j}\varepsilon_2 \partial_x Y_{j}\,dx-\int_{\R}\dot{z}_{j}\varepsilon_3 \partial_x Y_{j}\,dx-\int_{\R}\dot{z}_{j}\varepsilon_4 \partial_x Y_{j}\,dx.
					\end{aligned} 
				\end{equation*}
Then, combining the previous estimates obtained in \eqref{estimative2} and using the eigenvector construction of $H_j$ we have
$$\begin{aligned}
    \frac{d}{dt}\lambda_{j}^{\pm}(t)&=-\langle  \vec{\varepsilon} , H_j\vec{Y}_{j}^{\pm}\rangle +\|\vec{\varepsilon}\|_{X}^2+t^{-3}\\
						&=\epsilon_{0}^j \lambda_{j}^{\pm}(t)+\|\vec{\varepsilon}\|_{X}^2+t^{-3}.
\end{aligned}$$
                        Hence, the result follows.
\end{proof}
\subsection*{ Construction of a family of 2 -solitary waves.}

\begin{proposition}\label{prop1}
 There exist $n_{0}>0$ and $T_{0}>0$ large enough, such that for all $n \geq n_{0}$, there exist $\bar{z}_{n}>0,$  {$ |a_{k, n} |\leq T_{n}^{-3 / 2}$} with $k=1,2$ and  $\delta>0$ small enough, for any
\begin{equation}\label{condlast}
\left\{\begin{array}{l}
\left(z_{1,n}, z_{2,n}\right) \in \mathbb{R}^{2 } \text { with }\left|z_{1,n}-z_{2,n}\right|>\frac{5}{\sqrt{\omega_{\star}}}|\log \delta|   \\
\left|e^{\frac{1}{2} \bar{z}_{n}}-T_{n}\right|< \log T_{n}, \,\,  \left|z_{n}\left(T_{n}\right)\right|=\bar{z}_{n}, \quad \lambda_{k}^{-}(T_n)=a_{k,n} \\
 \lambda_{k}^{+}(T_n)=0 \quad 
 {\| \vec{\varepsilon}(T_n) \|_{X}\leq \delta}\text { with \eqref{2orthogonality} }
\end{array}\right.  
\end{equation}

there exists  {$|\left(\lambda_{1}^{-}(T_n), \lambda_{2}^{-}(T_n)\right)|\leq \delta^{5/4}$} such that the solution $\vec{u}_n$ of \eqref{system1} with the initial data
\begin{equation}\label{datalast}
\vec{u}_n(T_n)=\vec{R}_{1}(T_n)+\vec{R}_{2}(T_n)+\vec{W}\left(\lambda_{1}^{-}(T_n), \lambda_{2}^{-}(T_n)\right)    
\end{equation}
is a 2-solitary wave. This is, such that the corresponding solution $\vec{u}_{n}$ of \eqref{system1} exists on $\left[T_{0}, T_{n}\right]$, satisfies the decomposition of Lemma \ref{2lemamodulation},
\begin{equation}\label{boostrap1}
\vec{u}_{n}(t, x)=\sum_{k=1,2}\vec{R}_{k,n}(t)+\vec{\varepsilon}_n(t, x):=\sum_{k=1,2} T(\theta_k)\vec{\Phi}_{\omega_k}\left(x-z_{k, n}(t)\right)+\vec{\varepsilon}_n(t, x)    
\end{equation}
and verifies the following uniform estimates for all $t \in\left[T_{0}, T_{n}\right]$
 \begin{equation}\label{estiprop1}
\left|\left|z_{n}(t)\right|-2 \log (t)\right| \lesssim \log^{-1/2} t,  \quad\left\|\vec{\varepsilon}(t)\right\|_{X} \lesssim t^{-1} \log ^{-3 / 2} t .
\end{equation}
 \end{proposition}

For the sake of simplicity we drop the index $n$ (except for $T_{n}$ ) in the following sections. Our goal now is to prove Proposition \ref{prop1} using a bootstrap argument, integration of a differential system of geometrical parameters and energy estimates. Let $T_{0}>0$ independent of $n$ to be determined later and $\vec{u}$ be the corresponding solution of \eqref{system1}. We now define the maximal time interval $\left[T^{\star}, T_n\right]$ on which suitable exponential estimates hold.
			
			\begin{definition}\label{definia}
				Let $T^{\star}$ be the infimum of $T \geq T_{0}$ such that for all $t \in\left[T, T_n\right]$, both  \eqref{estiprop1} and \eqref{hipotese1} are hold.   
			\end{definition}

\subsection{ Bootstrap setting.}In this section we aim to improve by a certain factor the estimates presented in Proposition \ref{prop1}, see \eqref{estiprop1}. The proof of Proposition \ref{prop1} is based on the following bootstrap estimates, for $C \gg 1$ to be chosen later, we consider the following hypotheses: for all $t\in [T_0,T_n]$ assume \eqref{boostrap1} and 
\begin{equation}\label{hipotese1}
  \begin{aligned}
&\left| (\kappa g_0 )^{-1/2}e^{\frac{|z(t)|}{2}}-t\right|\leq \log^{-1/2}  t, 
\sum_{k=1,2}\left|\lambda_{k}^{+}(t)\right| \leq t^{-3 / 2}\quad \sum_{k=1,2}\left|\theta_{k}(t)\right| \leq t^{-1}, \\
&\sum_{k=1,2}\left|\lambda_{k}^{-}(t)\right| \leq t^{-3 / 2}, \quad\|\vec{\varepsilon}(t)\|_{X} \leq C^{\star} t^{-1} \log^{-3 / 2} t \, \, (see \, \, \eqref{estimative2}).
\end{aligned}  
\end{equation}
\begin{remark}\label{remarklast}
    \begin{enumerate}[(i)]

        \item Note that the estimate on $z$ gives us a more precise estimate,

\begin{equation*}
|z(t)|=2 \log t+O\left(\log ^{-1 / 2} t\right)`
\end{equation*}

where $C>0$ is a constant depending only on $p$. We can thus deduce,
$$
||z(t)|-2 \log t| \lesssim t^{-1} \log ^{-1 / 2} t.
$$
\item Note that from the assumption in \eqref{condlast} and \eqref{hipotese1} we have that for sufficiently large $T_n$, we have $\vec{\varepsilon}(T_n)=0$, which implies from \eqref{datalast} that $\vec{u}(T_n)=\vec{R}_1(T_n)+\vec{R}_2(T_n)$.
     \end{enumerate}  
\end{remark}
Let us then consider the following lemma where we establish some estimates for the modulated parameters

				\begin{lemma}\label{modu1}
For all $t \in\left[T_0, T_{n}\right]$, the following hold.
\begin{enumerate}[(i)]

    \item Estimates on $z_{k}$ and $z$. We have

\begin{equation*}
\sum_{k=1,2}\left|\dot{z}_{k}\right| \lesssim t^{-1} \log ^{-3/2} t, \quad \sum_{k=1,2}|\Phi_{k}(|z|)| \lesssim t^{-3} 
\end{equation*}

\item  Estimates on $\theta_{k}$. We have

\begin{equation*}
\sum_{k=1,2}| | \dot{\theta}_{k}\left|-t^{-2}\right| \lesssim t^{-2} \log ^{-1} t. 
\end{equation*}

\item Estimates on  {$\partial_{t} G$}.  We have
 
\begin{equation*}
\left\|\partial_{t} G\right\|_{L^{2}(\R)}\lesssim t^{-3}.
\end{equation*}
\end{enumerate}
\end{lemma} 

\begin{proof}
 
(i) From Lemma \ref{2lemamodulation}, for large enough $T_{0}$,
$$
\sum_{k=1,2}\left|\dot{z}_{k}\right| \lesssim\|\vec{\varepsilon}\|_{X} \lesssim t^{-1} \log ^{-3 / 2} t.
$$

Now, we consider $q$ such that $\Phi_{k}(x)=q_{\omega_k}(|x|)$. Then, observe that from \eqref{ground} and \eqref{hipotese1}, we have
$$
\left|q_{\omega_k}(|z|)\right| \lesssim|q_{\omega_k}(|z|)-\kappa e^{-\sqrt{1-\omega_{k}^2
}|z|}|+\kappa e^{-\sqrt{1-\omega_{k}^2
}|z|} \lesssim e^{-6\sqrt{\omega_{\star}}|z|}\lesssim t^{-3}.
$$
(ii) For the first estimate note that  from \eqref{estimative2.5}
$$
\begin{aligned}
\left|\dot{\theta}_{k}\mp g_{0} \frac{z}{|z|} q_{\omega_k}(|z|)\right| \lesssim & \left|\dot{\theta}_{k}\mp \frac{z}{|z|} g(|z|)\right|+\left|g(|z|)-g_{0} q_{\omega_k}(|z|)\right| \\
\lesssim & \|\vec{\varepsilon}\|_{X}^{2}+e^{-6\sqrt{\omega_{\star}}|z|}+|z|^{-1}e^{-\sqrt{\omega_{\star}}|z|} \\
\lesssim & \|\vec{\varepsilon}\|_{X}^{2}+t^{-2}+t^{-2} \log ^{-1} t \lesssim t^{-2} \log ^{-1} t.
\end{aligned}
$$
where $k=1,2$. Thus  
\begin{equation}\label{newtheta}
\left|\dot{\theta}_{k}\mp\kappa g_{0} \frac{z}{|z|}e^{-|z|} \right| 
\lesssim \left|\dot{\theta}_{k}\mp g_{0} \frac{z}{|z|} q(|z|)\right|+\left|q(|z|)-\kappa e^{-|z|} \right| \lesssim t^{-2} \log ^{-1} t.   
\end{equation}
Which in turn implies from \eqref{hipotese1}  that $$\sum_{k=1,2}\left||\dot{\theta}_k|-t^{-2}\right| \lesssim t^{-2} \log ^{-1} t.$$

(iii) By direct computation,
$$
\begin{aligned}
\partial_{t}G &=-f^{\prime}\left(R_{1}^{(1)}+R_{2}^{(1)}\right)\left(i\dot{\theta}_1 e^{i \theta_1}R_1^{(1)}+\dot{z}_{1}e^{i \theta_1}\partial_{x}R_{1}^{(1)}+\dot{z}_{2} e^{i \theta_2}  \partial_{x} R_{2}^{(1)}+i\dot{\theta}_2 e^{i \theta_2}R_{2}^{(1)}\right)\\
&\quad +\sum_{k=1,2} f^{\prime}\left(R_{k}^{(1)}\right)\left(\dot{z}_{k} \partial_{x} R_{k}^{(1)}-i\dot{\theta}_k e^{i \theta_k}R_{k}^{(1)}\right). 
\end{aligned}
$$

Thus we get
$$
\left\|\partial_{t} G\right\|_{L^{2}(\R)} \lesssim e^{-2\sqrt{\omega_{\star}}|z|} \sum_{k=1,2}\left(\left|\dot{z}_{k}\right|+ \left|\dot{\theta}_{k}\right|\right) \lesssim t^{-3}.
$$
and hence we obtain the desired inequality.
 
\end{proof}

Now, we will obtain some control estimates for the solitary waves that compose $\Vec{u}$. For this purpose, we will use an argument that involves localizing certain quantities. Therefore, 
			Let $\phi: [0,\infty) \rightarrow [0,\infty)$ be a $C^{\infty}$ cutoff function such that $\phi(s)=0$ for $s>\frac{1}{8}$, $\phi(s) \in [0,1]$ if $s \in [-1,1]$, and $\phi(s)=1$ for $<0s<\frac{1}{10}$. We define  for all $(x,t)\in [0,\infty) \times[0,\infty)$, 
	\[
			\begin{array}{ll}
				 \phi_{k}(x,t):=\phi\left(\frac{1}{\log t}\left|x -z_k(t)\right|\right), \, \,\text { for} \,\,k=1,2.
			\end{array}
			\]
            Note that, 

\begin{equation*}
\left|\partial_{t} \phi_{k}(t, x)\right| \lesssim \frac{\chi_{\Omega_{k}}(t, x)}{t \log t}, \quad \left|\partial_{x} \psi_{k}(t, x)\right| \lesssim \frac{\chi_{\Omega_{k}}(t, x)}{\log t} 
\end{equation*}

where
$$
\Omega_{k}(t, x)=\left\{x \in \mathbb{R}:\left|x-z_{k}(t)\right| \leq \frac{1}{8} \log t\right\}.
$$

 With this setup, we proceed to identify the quantities conserved by the flow. In fact, for $k=1,2,$ we define  
   \begin{equation*}\label{eneloc}
E_k(\Vec{u})=\int_{\R}\left(|u|^{2}+\left|\rho^{2}\right|+\left|u_{x}\right|^{2}+\alpha|u|^{2} v+\frac{\beta}{2}|u|^{4}+\frac{\alpha}{2} v^{2}+\frac{\alpha}{2} n^{2}\right)\phi_k (x) \, d x .\end{equation*}
\begin{equation*}\label{momloc1}
Q_{k}(\Vec{u})=2 \Im \int_{\R} \bar{u} \rho \phi_k(x) \,d x.
\end{equation*}
Finally, we define the operators 
\begin{equation*}\label{definitionSj}
    \mathcal{S}_{k,loc}(\Vec{u})=E_k(\Vec{u})-\omega_kQ_{k}(\Vec{u})
\end{equation*}
and 
\begin{equation*}\label{definitionS}
  \mathcal{S}(\Vec{u})=\sum_{k=1}^2\mathcal{S}_{k,loc}(\Vec{u}).  
\end{equation*}`

\begin{lemma}\label{taylorS}
There exists \(T_{0}\) such that if \(t_{0} > T_{0}\), then for all \(t \in [t_{0}, T^{n}]\),
\begin{equation*}\label{formSfinal}
        \mathcal{S}( \Vec{u}(t), t) = \sum_{j=1}^{2}\mathcal{S}_{j,loc}\left( \Vec{R}_{j}(t), t \right) + \mathcal{H}_{loc}( \vec{\varepsilon}(t), t) + O\left(t^{-3}\right),
\end{equation*}
where $$\begin{aligned}
    \mathcal{H}_{loc}( \vec{\varepsilon}(t), t)&=\int_{\R}|\partial_{x}\varepsilon_1|^2\,dx+\int_{\R}|\varepsilon_1|^2\,dx+\int_{\R}|\varepsilon_2|^2\,dx\\&\quad+\frac{\alpha}{2}\int_{\R}|\varepsilon_3|^2\,dx+\frac{\alpha}{2}\int_{\R}|\varepsilon_4|^2\,dx+\frac{\beta}{2}\int_{\R}|\varepsilon_1|^4\,dx\\
    &\quad +\sum_{j=1}^{2} \alpha \int_{\R}|\varepsilon_1|^2R_{j}^{(3)}\,dx+\sum_{j=1}^{2} 2\alpha \int_{\R} R_{j}^{(1)}\varepsilon_3\overline{\varepsilon}_1\,dx\\&\quad+\sum_{j=1}^{2} \alpha c_j \int_{\R} \varepsilon_3 \varepsilon_4 \Phi_j\,dx +\sum_{j=1}^{2} 2 \omega_j \int_{\R} \varepsilon_2 \overline{\varepsilon}_1\Phi_j\,dx \\
    &\quad +\sum_{j=1}^{2} 2 \beta \int_{\R} \Re(R_{j}^{(1)} \overline{\varepsilon}_1)R_{j}^{(1)} \overline{\varepsilon}_1\,dx.
\end{aligned}$$

\end{lemma}
\begin{proof} 
Let $t \in [t_{0},T^{n}]$ and fix $j \in\{1,2\}$. Observe that, from the definition \eqref{definitionSj} and \eqref{definitionS}, we have 
\begin{equation}\label{primerestima}
\begin{aligned}
\mathcal{S}( \Vec{u}) & = E(\vec{u})-\sum_{j=1}^{2} \omega_j Q_{j}(\vec{u}).
\end{aligned}
\end{equation}
Initially, we will estimate each term that composes $\mathcal{S}$ by taking $\vec{u}=\vec{R}+\vec{\varepsilon}=\sum_{j=1}^{2}\vec{R}_j+\vec{\varepsilon}.$

In fact, \begin{equation*}
    \begin{aligned}        E(\vec{u})&=E\left(\sum_{j=1}^{2}\vec{R}_j +\vec{\varepsilon}\right)\\&=\int_\rr\left|\sum_{s=1}^{2} \partial_{x} R_{s}^{(1)}+\partial_{x}\varepsilon_1\right|^{2} \,dx + \int_\rr\left|\sum_{k=1}^{2}  R_{k}^{(2)}+\varepsilon_2\right|^{2} \,dx+ \frac{\alpha}{2}\int_\rr\left|\sum_{k=1}^{2}  R_{k}^{(4)}+\varepsilon_4\right|^{2} \,dx \\
    & \quad + \frac{\beta}{2} \int_\rr\left|\sum_{k=1}^{2}  R_{k}^{(1)}+\varepsilon_1\right|^{4} \,dx + \alpha\Re \int_\rr\left(\sum_{k=1}^{2} R_{k}^{(3)}+\varepsilon_3\right)\left|\sum_{s=1}^{2} \bar{R}_{s}^{(1)}+\varepsilon_1\right|^{2} \,dx\\
    &\quad + \frac{\alpha}{2}\int_\rr\left|\sum_{k=1}^{2}  R_{k}^{(3)}+\varepsilon_3\right|^{2} \,dx+ \frac{\alpha}{2}\int_\rr\left|\sum_{k=1}^{2}  R_{k}^{(1)}+\varepsilon_1\right|^{2} \,dx.
    \end{aligned}
\end{equation*}

To estimate the terms of $E$, we will use Lemma \ref{solitons} part $3$, as follows: $$\begin{aligned}
    \int_\rr\left|\sum_{s=1}^{2} \partial_{x} R_{s}^{(1)}+\varepsilon_1\right|^{2} \,dx&=\int_\rr\left|\sum_{s=1}^{2} \partial_{x} R_{s}^{(1)}\right|^{2} \,dx+2\Re\int_{\R}\left(\sum_{s=1}^{2} \partial_{x} R_{s}^{(1)}\right)\partial_x \overline{\varepsilon}_1 \,dx+\int_{\R}\left|\partial_x\varepsilon_1\right|^{2} \,dx. 
\end{aligned}$$
We will expand the terms of the above integrals, so
$$
\begin{aligned}
 \left|\sum_{s=1}^{2} \partial_{x} R_{s}^{(1)}\right|^{2}&=\left(\partial_{x} R_{1}^{(1)}+\partial_{x} R_{2}^{(1)}\right)\left(\partial_{x} \bar{R}_{1}^{(1)}+\partial_{x} \bar{R}_{2}^{(1)}\right),
\end{aligned}
$$
whence,
\begin{equation*}
\left|\sum_{s=1}^{2} \partial_{x} R_{s}^{(1)}\right|^{2}=\sum_{s=1}^{2}\left|\partial_{x} R_{s}^{(1)}\right|^{2}+\sum_{\substack{s, m=1 \\ s \neq m}}^{2} \partial_{x} R_{s}^{(1)} \partial_{x} \bar{R}_{j}^{(1)}.
\end{equation*} 
Therefore, from Lemma \ref{solitons},$$
\begin{aligned}
\int_\rr\left|\sum_{s=1}^{2} \partial_{x} R_{s}^{(1)}+\varepsilon_1\right|^{2} \,dx
    &=\int_\rr\left|\sum_{s=1}^{2} \partial_{x} R_{s}^{(1)}\right|^{2} \,dx+2\Re\int_{\R}\left(\sum_{s=1}^{2} \partial_{x} R_{s}^{(1)}\right)\partial_x \overline{\varepsilon}_1 \,dx+\int_{\R}\left|\partial_x\varepsilon_1\right|^{2} \,dx\\
    &=\sum_{j=1}^{2} \int_{\R}\left|\partial_{x} R_{j}^{(1)}\right|^{2} \,dx-\sum_{j=1}^{2}\Re\int_{\R}  2\partial_{xx}^{2} R_{j}^{(1)} \overline{\varepsilon}_1 \,dx+\int_{\R}\left|\partial_{x}\varepsilon_1\right|^{2} \,dx\\
    &\quad+O\left(t^{-3}\right).
    \end{aligned}$$
Similarly to the previous computations, we have, $$\begin{aligned}
    \int_\rr\left|\sum_{s=1}^{2} R_{s}^{(2)}+\varepsilon_2\right|^{2} \,dx&=\int_\rr\left|\sum_{s=1}^{2}  R_{s}^{(2)}\right|^{2} \,dx+2\Re\int_{\R}\left(\sum_{s=1}^{2} R_{s}^{(2)}\right) \overline{\varepsilon}_2 \,dx+\int_{\R}\left|\varepsilon_2\right|^{2} \,dx \\
    &=\sum_{j=1}^{2} \int_{\R}\left| R_{j}^{(2)}\right|^{2} \,dx+\sum_{j=1}^{2}\Re\int_{\R}  2 R_{j}^{(2)} \overline{\varepsilon}_2 \,dx+\int_{\R}\left|\varepsilon_2\right|^{2} \,dx+O\left(t^{-3}\right),
\end{aligned}$$
$$\begin{aligned}
    \int_\rr\left|\sum_{s=1}^{2} R_{s}^{(1)}+\varepsilon_1\right|^{2} \,dx&=\int_\rr\left|\sum_{s=1}^{2}  R_{s}^{(1)}\right|^{2} \,dx+2\Re\int_{\R}\left(\sum_{s=1}^{2} R_{s}^{(1)}\right) \overline{\varepsilon}_1 \,dx+\int_{\R}\left|\varepsilon_1\right|^{2} \,dx \\
    &=\sum_{j=1}^{2} \int_{\R}\left| R_{j}^{(1)}\right|^{2} \,dx+\sum_{j=1}^{2}\Re\int_{\R}  2 R_{j}^{(1)} \overline{\varepsilon}_1 \,dx+\int_{\R}\left|\varepsilon_1\right|^{2} \,dx+O\left(t^{-3}\right), 
\end{aligned}$$
$$\begin{aligned}
    \frac{\alpha}{2}\int_\rr\left(\sum_{s=1}^{2} R_{S}^{(3)}+\varepsilon_3\right)^{2} \,dx&=\frac{\alpha}{2}\int_\rr\left(\sum_{s=1}^{2}  R_{S}^{(3)}\right)^{2} \,dx+2\frac{\alpha}{2}\Re\int_{\R}\left(\sum_{s=1}^{2} R_{S}^{(3)}\right) \overline{\varepsilon}_3 \,dx+\frac{\alpha}{2}\int_{\R}\left(\varepsilon_3\right)^{2} \,dx \\
    &=\frac{\alpha}{2}\sum_{j=1}^{2} \int_{\R}\left( R_{j}^{(3)}\right)^{2} \,dx+\sum_{j=1}^{2}\Re\int_{\R}  \alpha R_{j}^{(3)} \overline{\varepsilon}_3 \,dx+\frac{\alpha}{2}\int_{\R}\left(\varepsilon_1\right)^{2} \,dx\\
&\quad+O\left(t^{-3}\right)
\end{aligned}$$
and $$\begin{aligned}
    &\frac{\alpha}{2}\int_\rr\left(\sum_{s=1}^{2} R_{S}^{(4)}+\varepsilon_4\right)^{2} \,dx\\&=\frac{\alpha}{2}\int_\rr\left(\sum_{s=1}^{2}  R_{S}^{(4)}\right)^{2} \,dx+2\frac{\alpha}{2}\Re\int_{\R}\left(\sum_{s=1}^{2} R_{S}^{(4)}\right) \overline{\varepsilon}_4 \,dx+\frac{\alpha}{2}\int_{\R}\left(\varepsilon_4\right)^{2} \,dx \\
    &=\frac{\alpha}{2}\sum_{j=1}^{2} \int_{\R}\left( R_{j}^{(4)}\right)^{2} \,dx+\sum_{j=1}^{2}\Re\int_{\R}  \alpha R_{j}^{(4)} \overline{\varepsilon}_4\,dx+\frac{\alpha}{2}\int_{\R}\left(\varepsilon_4\right)^{2} \,dx+O\left(t^{-3}\right).
\end{aligned}$$
Now, we have 
$$\begin{aligned}
    &\alpha\int_\rr\left(\sum_{s=1}^{2} R_{S}^{(3)}+\varepsilon_3\right)\left|\sum_{s=1}^{2} R_{s}^{(1)}+\varepsilon_1\right|^2\,dx\\&=\alpha\int_\rr\left(\sum_{s=1}^{2}  R_{S}^{(3)}\right) \left|\sum_{s=1}^{2} R_{s}^{(1)}\right|^2\,dx+\alpha\int_\rr\left(\sum_{s=1}^{2}  R_{S}^{(3)}\right) \left|\varepsilon_1\right|^2\,dx\\&\quad+\alpha\int_\rr\left|\sum_{s=1}^{2}  R_{s}^{(1)}\right|^2 \varepsilon_3\,dx+\alpha\int_\rr\varepsilon_3 |\varepsilon_1|^2\,dx\\
    &\quad+\alpha \int_{\R}\left( \sum_{j=1}^{2}R_{j}^{(3)}\right) 2\Re \left(\sum_{k=1}^{2}R_{k}^{(1)} \overline{\varepsilon}_1\right)\,dx+\alpha \int_{\R} 2\Re \left(\sum_{k=1}^{2}R_{k}^{(1)} \overline{\varepsilon}_1\right)\varepsilon_3\,dx.
\end{aligned}$$
Furthermore, note that 
\begin{equation}\label{somadouble}
\begin{aligned}
&\left(\sum_{k=1}^{2} R_{k}^{(3)}\right)\left|\sum_{s=1}^{2} \bar{R}_{S}^{(1)}\right|^{2}\\ &=R_{j}^{(3)}\left(\bar{R}_{j}^{(1)}\right)^{2}+2\sum_{\substack{s=1\\ s\neq j}}^{N} \bar{R}_{S}^{(1)}\bar{R}_{j}^{(1)}R_{j}^{(3)}+ \sum_{\substack{k=1\\k\neq j}}^{N} R_{k}^{(3)}\,(\bar{R}_{j}^{(1)})^{2} + \left(\sum_{\substack{s=1\\s\neq j}}^{N} \bar{R}_{S}^{(1)}\right)^{2}R_{j}^{(3)}\\
&\quad +2\left(\sum_{\substack{k=1\\k\neq j}}^{N} R_{k}^{(3)}\right)\left(\sum_{\substack{s=1\\s\neq j}}^{N} \bar{R}_{S}^{(1)}\right)\bar{R}_{j}^{(1)}+\left(\sum_{\substack{k=1\\k\neq j}}^{N} R_{k}^{(3)}\right)\left(\sum_{\substack{s=1\\s\neq j}}^{N} \bar{R}_{S}^{(1)}\right)^{2}.
\end{aligned}
\end{equation}
Thus, using Lemma \ref{solitons} along with the above estimates, we obtain
$$\begin{aligned}
    &\alpha\int_\rr\left(\sum_{s=1}^{2} R_{S}^{(3)}+\varepsilon_3\right)\left|\sum_{s=1}^{2} R_{s}^{(1)}+\varepsilon_1\right|^2\,dx\\&=\sum_{j=1}^{2}\int_\rr\alpha\left(  R_{j}^{(3)}\right) \left| R_{j}^{(1)}\right|^2\,dx+\sum_{j=1}^{2}\int_\rr\alpha\left(  R_{j}^{(3)}\right) \left|\varepsilon_1\right|^2\,dx+\sum_{j=1}^{2}\int_\rr \alpha\left| R_{j}^{(1)}\right|^2 \varepsilon_3\,dx\\&\quad+\alpha\int_\rr\varepsilon_3 |\varepsilon_1|^2\,dx+ \sum_{j=1}^{2}\int_{\R}\left( R_{j}^{(3)}\right) 2\alpha\Re \left(R_{j}^{(1)} \overline{\varepsilon}_1\right)\,dx+\sum_{j=1}^{2}\int_{\R} 2\alpha \Re \left(R_{j}^{(1)} \overline{\varepsilon}_1\right)\varepsilon_3\,dx
\\&\quad +O\left(t^{-3}\right).
\end{aligned}$$
Finally, it only remains to analyze $$\begin{aligned}
   & \frac{\beta}{2}\int_\rr\left|\sum_{s=1}^{2} R_{s}^{(1)}+\varepsilon_1\right|^4\,dx=\frac{\beta}{2}\int_\rr\left(\left|\sum_{s=1}^{2} R_{s}^{(1)}\right|^2+2\Re\left(\sum_{s=1}^{2} R_{s}^{(1)} \overline{\varepsilon}_1\right)+|\varepsilon_1|^2\right)^2\,dx\\
    &=\frac{\beta}{2}\int_\rr\left(\left|\sum_{s=1}^{2} R_{s}^{(1)}\right|^2+2\Re\left(\sum_{s=1}^{2} R_{s}^{(1)} \overline{\varepsilon}_1\right)\right)^2\,dx+\beta\int_\rr\left(\left|\sum_{s=1}^{2} R_{s}^{(1)}\right|^2+2\Re\left(\sum_{s=1}^{2} R_{s}^{(1)} \overline{\varepsilon}_1\right)\right)|\varepsilon_1|^2\,dx\\
    &\quad+\frac{\beta}{2}\int_\rr|\varepsilon_1|^4\,dx\\
    &=\frac{\beta}{2}\int_\rr\left|\sum_{s=1}^{2} R_{s}^{(1)}\right|^4\,dx+2\beta\int_\rr\left|\sum_{s=1}^{2} R_{s}^{(1)}\right|^2\Re\left(\sum_{s=1}^{2} R_{s}^{(1)} \overline{\varepsilon}_1 \right)\,dx+2\beta\int_\rr\Re\left(\sum_{s=1}^{2} R_{s}^{(1)} \overline{\varepsilon}_1 \right)^2\,dx\\
    &\quad+\beta\int_\rr\left|\sum_{s=1}^{2} R_{s}^{(1)}\right|^2|\varepsilon_1|^2\,dx +2\beta\int_\rr\Re\left(\sum_{s=1}^{2} R_{s}^{(1)} \overline{\varepsilon}_1 \right)|\varepsilon_1|^2\,dx+\frac{\beta}{2}\int_\rr|\varepsilon_1|^4\,dx.
\end{aligned}$$
Following the same steps as in the previous estimates and applying Lemma \ref{solitons}, we obtain
$$\begin{aligned}
   & \frac{\beta}{2}\int_\rr\left|\sum_{s=1}^{2} R_{s}^{(1)}+\varepsilon_1\right|^4\,dx\\&=\sum_{j=1}^{2}\int_\rr\frac{\beta}{2}\left| R_{j}^{(1)}\right|^4\,dx+\sum_{s=1}^{2}\int_\rr2\beta\Re\left( R_{j}^{(1)} \overline{\varepsilon}_1 \right)^2\,dx+\frac{\beta}{2}\int_\rr|\varepsilon_1|^4\,dx+O\left(t^{-3}\right).
\end{aligned}$$
Combining the previous estimates, we have
\begin{equation*}
    \begin{aligned}        &E(\vec{u})\\&=\sum_{j=1}^{2} \int_{\R}\left|\partial_{x} R_{j}^{(1)}\right|^{2} \,dx-\sum_{j=1}^{2}\Re\int_{\R}  2\partial_{xx}^{2} R_{j}^{(1)} \overline{\varepsilon}_1 \,dx+\int_{\R}\left|\partial_{x}\varepsilon_1\right|^{2} \,dx+\sum_{j=1}^{2} \int_{\R}\left| R_{j}^{(1)}\right|^{2} \,dx +\sum_{j=1}^{2}\Re\int_{\R}  2 R_{j}^{(1)} \overline{\varepsilon}_1 \,dx\\&\quad+\int_{\R}\left|\varepsilon_1\right|^{2} \,dx  + \sum_{j=1}^{2} \int_{\R}\left| R_{j}^{(2)}\right|^{2} \,dx+\sum_{j=1}^{2}\Re\int_{\R}  2 R_{j}^{(2)} \overline{\varepsilon}_2 \,dx+\int_{\R}\left|\varepsilon_2\right|^{2} \,dx +\frac{\alpha}{2}\sum_{j=1}^{2} \int_{\R}\left( R_{j}^{(3)}\right)^{2} \,dx\\
    &\quad+\sum_{j=1}^{2}\Re\int_{\R}  \alpha R_{j}^{(3)} \overline{\varepsilon}_3 \,dx+\frac{\alpha}{2}\int_{\R}\left(\varepsilon_1\right)^{2} \,dx +\frac{\alpha}{2}\sum_{j=1}^{2} \int_{\R}\left( R_{j}^{(4)}\right)^{2} \,dx+\sum_{j=1}^{2}\Re\int_{\R}  \alpha R_{j}^{(4)} \overline{\varepsilon}_4 \,dx+\frac{\alpha}{2}\int_{\R}\left(\varepsilon_4\right)^{2} \,dx\\
&\quad+\sum_{j=1}^{2}\int_\rr\alpha\left(  R_{j}^{(3)}\right) \left| R_{j}^{(1)}\right|^2\,dx+\sum_{j=1}^{2}\int_\rr\alpha\left(  R_{j}^{(3)}\right) \left|\varepsilon_1\right|^2\,dx+\sum_{j=1}^{2}\int_\rr \alpha\left| R_{j}^{(1)}\right|^2 \varepsilon_3\,dx\\&\quad+ \sum_{j=1}^{2}\int_{\R}\left( R_{j}^{(3)}\right) 2\alpha\Re \left(R_{j}^{(1)} \overline{\varepsilon}_1\right)\,dx+\sum_{j=1}^{2}\int_{\R} 2\alpha \Re \left(R_{j}^{(1)} \overline{\varepsilon}_1\right)\varepsilon_3\,dx+\sum_{s=1}^{2}\int_\rr2\beta\Re\left( R_{j}^{(1)} \overline{\varepsilon}_1 \right)^2\,dx\\&\quad+\frac{\beta}{2}\int_\rr|\varepsilon_1|^4\,dx+\sum_{j=1}^{2}\int_\rr\frac{\beta}{2}\left| R_{j}^{(1)}\right|^4\,dx+O\left(t^{-3}\right).
    \end{aligned}
\end{equation*}
That is, \begin{equation*}\label{1accou}
    \begin{aligned}        E(\vec{u})&=\sum_{j=1}^{2}E(\vec{R}_j) -\sum_{j=1}^{2}\Re\int_{\R}  2\partial_{xx}^{2} R_{j}^{(1)} \overline{\varepsilon}_1 \,dx+\int_{\R}\left|\partial_{x}\varepsilon_1\right|^{2} \,dx +\sum_{j=1}^{2}\Re\int_{\R}  2 R_{j}^{(1)} \overline{\varepsilon}_1 \,dx\\&\quad+\int_{\R}\left|\varepsilon_1\right|^{2} \,dx  +\sum_{j=1}^{2}\Re\int_{\R}  2 R_{j}^{(2)} \overline{\varepsilon}_2 \,dx+\int_{\R}\left|\varepsilon_2\right|^{2} \,dx+\sum_{j=1}^{2}\Re\int_{\R}  \alpha R_{j}^{(3)} \overline{\varepsilon}_3 \,dx\\
    &\quad+\frac{\alpha}{2}\int_{\R}\left(\varepsilon_1\right)^{2} \,dx+\sum_{j=1}^{2}\Re\int_{\R}  \alpha R_{j}^{(4)} \overline{\varepsilon}_4 \,dx+\frac{\alpha}{2}\int_{\R}\left(\varepsilon_4\right)^{2} \,dx+\sum_{j=1}^{2}\int_\rr\alpha\left(  R_{j}^{(3)}\right) \left|\varepsilon_1\right|^2\,dx\\
    &\quad+\sum_{j=1}^{2}\int_\rr \alpha\left| R_{j}^{(1)}\right|^2 \varepsilon_3\,dx+ \sum_{j=1}^{2}\int_{\R}\left( R_{j}^{(3)}\right) 2\alpha\Re \left(R_{j}^{(1)} \overline{\varepsilon}_1\right)\,dx+\sum_{j=1}^{2}\int_{\R} 2\alpha \Re \left(R_{j}^{(1)} \overline{\varepsilon}_1\right)\varepsilon_3\,dx\\&\quad+\sum_{s=1}^{2}\int_\rr2\beta\Re\left( R_{j}^{(1)} \overline{\varepsilon}_1 \right)^2\,dx+\frac{\beta}{2}\int_\rr|\varepsilon_1|^4\,dx+O\left(t^{-3}\right).
    \end{aligned}
\end{equation*}
Finally, by performing a process very similar to the previous one, we can obtain that 
\begin{equation}\label{accou3}
    \begin{aligned}
    -\sum_{j=1}^{2} \omega_j Q_{j}(\vec{u})
    &= -\sum_{j=1}^{2} \omega_j Q_{j}(\vec{R}_j)-\sum_{j=1}^{2} \omega_j \bigg[ \Im\int_{\R} 2   \overline{R_{j}^{(1)}}\varepsilon_2R_{j}\,dx+\Im \int_{\R} 2 R_{j}^{(2)}\overline{\varepsilon}_1R_{j}\,dx+\int_{\R}2\overline{\varepsilon}_1\varepsilon_2R_{j}\,dx\\
    &\quad \quad \quad \quad \quad \quad + O\left(t^{-3}\right)\bigg].
\end{aligned}
\end{equation}
Using the previous estimates, \eqref{1accou}-\eqref{accou3}, along with the fact that $\vec{R}_j$ satisfies \eqref{pointcrit} for each $j$, it follows that for all $j \in \{1,2\},$
$$
\mathcal{S}( \Vec{u}(t), t) = \sum_{j=1}^{2}\left(E(\vec{R}_j)-\omega_jQ_{j}(\vec{R}_j)\right) + \mathcal{H}_{loc}( \vec{\varepsilon}(t), t) + O\left(t^{-3}\right),
$$
where $$\begin{aligned}
    \mathcal{H}_{loc}( \vec{\varepsilon}(t), t)&=\int_{\R}|\partial_{x}\varepsilon_1|^2\,dx+\int_{\R}|\varepsilon_1|^2\,dx+\int_{\R}|\varepsilon_2|^2\,dx\\&\quad+\frac{\alpha}{2}\int_{\R}|\varepsilon_3|^2\,dx+\frac{\alpha}{2}\int_{\R}|\varepsilon_4|^2\,dx+\frac{\beta}{2}\int_{\R}|\varepsilon_1|^4\,dx\\
    &\quad +\sum_{j=1}^{2} \alpha \int_{\R}|\varepsilon_1|^2R_{j}^{(3)}\,dx+\sum_{j=1}^{2} 2\alpha \int_{\R} R_{j}^{(1)}\varepsilon_3\overline{\varepsilon}_1\,dx\\&\quad+\sum_{j=1}^{2} 2 \omega_j \int_{\R} \varepsilon_2 \overline{\varepsilon}_1\Phi_j\,dx +\sum_{j=1}^{2} 2 \beta \int_{\R} \Re(R_{j}^{(1)} \overline{\varepsilon}_1)R_{j}^{(1)} \overline{\varepsilon}_1\,dx.
\end{aligned}$$
Finally, note that $$E_j(\vec{R}_j)=E(\vec{R}_j)+ O\left(t^{-3}\right),$$
which is due to the definition of $\Phi_j$ for each $j$. For example, we can take a component of the sums in $E_j$ and note that: $$\int_\rr\left|\sum_{s=1}^{2} \partial_{x} R_{s}^{(1)}\right|^{2} R_{j}\,dx=\int_\rr\left|\partial_{x} R_{j}^{(1)}\right|^{2} \,dx-\int_\rr\left|\partial_{x} R_{j}^{(1)}\right|^{2}\sum_{\substack{k=1 \\ k \neq j}}^{N}R_{k} \,dx,$$
Thus, from Lemma \ref{solitons}, the assertion follows.
\end{proof}

We now propose a new way to write the previous lemma.
\begin{lemma}\label{newS}
    There exists \(T_{0}\) such that if \(t_{0} > T_{0}\), then for all \(t \in [t_{0}, T^{n}]\),
$$
\mathcal{S}( \Vec{u}(t), t) = \mathcal{S}\left( \Vec{R}(t), t \right) + \mathcal{H}_{loc}( \vec{\varepsilon}(t), t) +O\left(t^{-3}\right).
$$
\end{lemma}
\begin{proof}
    Let us $t \in [t_{0}, T^{n}].$ To prove this result, from the previous lemma, it would suffice to show that
    $$\mathcal{S}\left( \Vec{R}(t), t \right)=\sum_{j=1}^{2}\mathcal{S}_{j,loc}\left( \Vec{R}_j(t), t \right)+ O\left(t^{-3}\right).$$
    This follows as in the previous lemma. In fact, let us show some steps with the components of $S$,
\begin{equation*}\label{enesoma}
\begin{aligned}
     E\left(\Vec{R}(t),t\right)  
    & = \int_\rr\left|\sum_{s=1}^{2} \partial_{x} R_{s}^{(1)}\right|^{2} \,dx + \int_\rr\left|\sum_{k=1}^{2}  R_{k}^{(2)}\right|^{2} \,dx+ \frac{\alpha}{2}\int_\rr\left|\sum_{k=1}^{2}  R_{k}^{(4)}\right|^{2} \,dx \\
    & \quad + \frac{\beta}{2} \int_\rr\left|\sum_{k=1}^{2}  R_{k}^{(1)}\right|^{4} \,dx + \alpha\Re \int_\rr\left(\sum_{k=1}^{2} R_{k}^{(4)}\right)\left(\sum_{s=1}^{2} \bar{R}_{s}^{(1)}\right)^{2} \,dx\\
    &\quad + \frac{\alpha}{2}\int_\rr\left|\sum_{k=1}^{2}  R_{k}^{(3)}\right|^{2} \,dx+ \int_\rr\left|\sum_{k=1}^{2}  R_{k}^{(1)}\right|^{2} \,dx
\end{aligned}
\end{equation*}
We will expand the terms of the above integrals, 
$$
\begin{aligned}
 \left|\sum_{s=1}^{2} \partial_{x} R_{s}^{(1)}\right|^{2}&=\left(\partial_{x}R_{1}^{(1)}+\partial_{x} R_{2}^{(1)}\right)\left(\partial_{x} \bar{R}_{1}^{(1)}+\partial_{x} \bar{R}_{2}^{(1)}\right),
\end{aligned}
$$
we have
\begin{equation*}\label{somagradiente}
\begin{aligned}
 \left|\sum_{s=1}^{2} \partial_{x} R_{s}^{(1)}\right|^{2}&=\sum_{s=1}^{2}\left|\partial_{x} R_{s}^{(1)}\right|^{2}+\sum_{\substack{s, m=1 \\ s \neq m}}^{2} \partial_{x} R_{s}^{(1)} \partial_{x} \bar{R}_{j}^{(1)} \\&= \sum_{s=1}^{2}\left|\partial_{x} R_{s}^{(1)}\right|^{2}R_{S}+\sum_{s=1}^{2}\left|\partial_{x} R_{s}^{(1)}\right|^{2}\sum_{\substack{k=1 \\ k\neq s}}^{2}R_{k}+\sum_{\substack{s, m=1 \\ s \neq m}}^{2} \partial_{x} R_{s}^{(1)} \partial_{x} \bar{R}_{j}^{(1)}  
\end{aligned}
\end{equation*}

Now, using the fact that the solitons are bounded (see Proposition \ref{decaimentoquadratico}) and Lemma \ref{solitons}, we obtain
\begin{equation*}\label{somadeorden4}
\begin{aligned}
&\sum_{s=1}^{2}\int_{\R}\left|\partial_{x} R_{s}^{(1)}\right|^{2}\sum_{\substack{k=1 \\ k\neq s}}^{2}R_{k}\,dx+\sum_{\substack{s, m=1 \\ s \neq m}}^{2} \int_{\mathbb{R}} \partial_{x} R_{s}^{(1)} \partial_{x} \bar{R}_{j}^{(1)}R_{j}\,dx\\
&\qquad\leq Ce^{-6\sqrt{\omega_{\star}}|z|}.
\end{aligned}
\end{equation*}
Taking $T_{0}$ sufficiently large such that 
\begin{equation*}\label{T_{0}3}
e^{-6\sqrt{\omega_{\star}}|z|}\leq t^{-3},  
\end{equation*} 
we have that
\begin{equation*}\label{resulsoma}
\begin{aligned}
\sum_{s=1}^{2}\int_{\R}\left|\partial_{x} R_{s}^{(1)}\right|^{2}\sum_{\substack{k=1 \\ k\neq s}}^{2}R_{k}\,dx+\sum_{\substack{s, m=1 \\ s \neq m}}^{2} \int_{\mathbb{R}} \partial_{x} R_{s}^{(1)} \partial_{x} \bar{R}_{j}^{(1)}R_{j}\,dx\leq Ct^{-3},
\end{aligned}
\end{equation*}
from where $$\int_\rr\left|\sum_{s=1}^{2} \partial_{x} R_{S}^{(1)}\right|^{2} \,dx=\sum_{j=1}^{2}\int_\rr\left| \partial_{x} R_{j}^{(1)}\right|^{2} \Phi_j\,dx+Ce^{-3\sqrt{\omega_{\star}}c_{\star}t}.$$
Thus, following the same process, we can write that
\begin{equation*}
\begin{aligned}
E\left(\Vec{R}(t),t\right)
& =\sum_{j=1}^{2}\bigg(\int_{\mathbb{R}}\left|\partial_{x} R_{j}^{(1)}\right|^{2} R_{j}\,dx+ \int\left|\partial_{x} R_{j}^{(1)}\right|^{2} R_{j} \,dx+ \int\left|\partial_{x} R_{j}^{(2)}\right|^{2} R_{j} \,dx \\
&\quad+ \frac{\alpha}{2}\int\left| R_{j}^{(3)}\right|^{2} R_{j} \,dx+ \frac{\alpha}{2}\int\left| R_{j}^{(4)}\right|^{2} R_{j} \,dx+\alpha \Re \int_{\mathbb{R}} R_{j}^{(3)}\bar{R}_{j}^{(1)} R_{j}\,dx\\
&\quad+ \frac{\beta}{2}\int\left| R_{j}^{(1)}\right|^{4} R_{j} \,dx\bigg)+\mathcal{O}\left(t^{-3}\right).
\end{aligned}  
\end{equation*}
Finally, proceeding similarly, we obtain that 
\begin{equation*}\label{m1}
 \begin{aligned}
   Q_2(\Vec{R}(t),t) =\sum_{j=1}^{2}\Im\int_{\mathbb{R}}2R_{j}^{(2)}\overline{R_{j}^{(1)}} R_{j}\,dx+O\left(t^{-3}\right).
\end{aligned}   
\end{equation*}
It follows that, $$\mathcal{S}\left( \Vec{R}(t), t \right)=\sum_{j=1}^{2}\mathcal{S}_{j,loc}\left( \Vec{R}_j(t), t \right)+ O\left(t^{-3}\right).$$
From which we deduce that the expression \eqref{formSfinal} can be written as
$$
\mathcal{S}( \Vec{u}(t), t) = \mathcal{S}\left( \Vec{R}(t), t \right) + \mathcal{H}_{loc}( \vec{\varepsilon}(t), t) +O\left(t^{-3}\right).
$$
\end{proof}
\begin{lemma}\label{VariacionS}
    If $t_{0} >T_{0}$, there exists $C >0$ independent of $n$  such that, for all $t \in [t_{0},T^{n}]$,
$$
\left|\frac{\partial \mathcal{S}(t, \Vec{u}(t))}{\partial t}\right| \leqslant \frac{C}{t^3} \log^{-3} t.
$$
\end{lemma}
\begin{proof}
We observe that for all $t \in [t_{0},T^{n}]$,
$$
\mathcal{S} (\vec{w}(t),t)=E(\vec{w}(t))-\omega_{j} Q_{j}( \vec{w}(t),t)
$$
Since the energy $E$ is conserved by the flow of \eqref{system2}, to estimate the variations of $\mathcal{S}(t, \Vec{u}(t))$ it is enough to study the variations of the localized momentums $Q_{j}( \Vec{u}(t),t)$. 

For $j\in \{1,2\}$, we have 

$$
\begin{aligned}
 &\frac{\partial}{\partial t}\int_{\R}Q_{j}(\Vec{u}(t))dx\\
&= 2 \frac{\partial}{\partial t} \Im \int_{\R} \bar{u}(t) \rho(t) \psi_{j}(x,t)\, d x\\
& =2\Im\int_{\R} \frac{\partial}{\partial t} \bar{u}(t) \rho(t) \psi_{j}(x,t) \, d x+2\Im\int_{\R} \bar{u}(t)\frac{\partial}{\partial t}\rho(t) \psi_{j}(x,t)  \,d x+2\Im\int_{\R} \bar{u}(t)\rho(t)\frac{\partial}{\partial t} \psi_{j}(x,t)  \,d x.
\end{aligned}$$
Now, we observe that 
\begin{equation*}\label{Q0}
    \begin{aligned}
      2\Im\int_{\R} \bar{u}(t)\rho(t)\frac{\partial}{\partial t} \psi_{j}(x,t)  \,d x &=2\Im\int_{\R} \bar{u}(t)\rho(t)\left(\frac{x-z_k(t)}{\log t}\dot{z}_k\right)\psi_{j}^{\prime}(x,t)   \,d x,
    \end{aligned}
\end{equation*}
Similarly,
\begin{equation*}\label{Q1}
    \begin{aligned}
        2\Im \int_{\R} \frac{\partial}{\partial t} u(t) \rho(t) \psi_{j}(x,t) \, d x&=-2\Im\int_{\R} \bar{\rho}(t) \rho(t) \psi_{j}(x,t) \, d x \\
        &=-2\Im\int_{\R}  |\rho(t)|^2 \psi_{j}(x,t) \, d x\\
        &=0.
    \end{aligned}
\end{equation*}
Finally, \begin{equation*}\label{Q2}
    \begin{aligned}
        2\Im\int_{\R} \bar{u}(t)\frac{\partial}{\partial t}\rho(t) \psi_{j}(x,t)  \,d x&=2\Im\int_{\R} \bar{u}(t)\left(-u_{xx}+u+\alpha uv+\beta|u|^2 u\right) \psi_{j}(x,t)  \,d x \\&=2\Im\int_{\R} \bar{u}(t)\partial_{xx}u(t)\psi_{j}(x,t)  \,d x+2\Im\int_{\R} \bar{u}(t)u(t)\psi_{j}(x,t)  \,d x \\&\quad+2\alpha\Im\int_{\R} \bar{u}(t)u(t)v(t)\psi_{j}(x,t)  \,d x+2\Im\int_{\R} \bar{u}(t)|u|^2(t)u(t)\psi_{j}(x,t)  \,d x\\
        &=\frac{|\dot{z}_k|}{t^2 \log^2 t}\Im\int_{\R} \bar{u}(t)u(t)\psi_{j}^{\prime \prime}(x,t)  \,d x.
    \end{aligned}
\end{equation*}
Therefore, 
$$\left|\frac{\partial}{\partial t}\int_{\R}Q_{j}(\Vec{u}(t))dx\right|\leq \frac{C|\dot{z}_k|}{t^2 \log^2 t}\int_{\R}|u(t)|^2\psi_{j}^{\prime \prime}(x,t) \, d x.$$

Remembering  $
\Omega_{k}(t, x)=\left\{x \in \mathbb{R}:\left|x-z_{k}(t)\right| \leq \frac{1}{8} \log t\right\}.
,$ then  $$\left|\frac{\partial}{\partial t}\int_{\R}Q_{j}(\Vec{u}(t))dx\right|\leq \frac{C|\dot{z}_k|}{t^2 \log^2 t}\int_{\R}|u(t)|^2\psi_{j}^{\prime \prime}(x,t) \, d x.$$
Now, as for $j \in \{1,2\}$ we have
\begin{equation*}\label{termomomento1}
\begin{aligned}
 \int_{\Omega_{j}}|u(t)|^{2}   \,dx \leqslant \int_{\Omega_{j}}|R_{1}^{(1)}(t)|^{2}\,dx + \|\Vec{u}(t) - \vec{R}(t)\|_{X}^{2}   \leqslant \int_{\Omega_{j}}|R_{1}^{(1)}(t)|^{2}\,dx + \frac{C}{t^2 \log^3 t}.
\end{aligned}
\end{equation*}
Now, let us estimate the terms involving $R_{1}^{(1)}$.  Note that, by using the soliton decay (Proposition \ref{decaimentoquadratico}), we have
$$\int_{\Omega_{j}} |R_{1}^{(1)}|^2 \,dx\leqslant \int_{\Omega_{j}}e^{\beta |x - z_k(t)|} \, dx, \, \, \, for \, \, \, some \, \, 0<\beta <1.$$
then, we will study

\[
 \int_{|x - z_j(t)| < \frac{1}{8}\log(t)} e^{-\beta |x - z_j(t)|} \, dx.
\]

Taking the change of variable $ u = x - z_j(t) $, we obtain

\[
 \int_{|x - z_j(t)| < \frac{1}{8}\log(t)} e^{-\beta |x - z_j(t)|} \, dx=2 \int_0^{\frac{1}{8} \log(t)} e^{\beta u} \, du 
= \frac{2}{\beta} \left( t^{\frac{1}{8}\beta } - 1 \right)
= \frac{2}{|\beta|} \left(1 - t^{-\frac{1}{8}|\beta| } \right).
\]
Hence, 
\[
\int_{|x - z_j(t)| < \frac{1}{8}\log(t)} e^{-\beta |x - z_j(t)|} \, dx \sim \frac{2}{|\beta|}, \quad \text{since}\, \,  t^{-\frac{1}{8}|\beta| } \to 0.
\]
It follows that for $j \in \{1,2\}$, 
\begin{equation*}\label{varmomem}
\left|\int_{\Omega_{j}} |R_{1}^{(1)}|^2 \,dx\right|\leq C
\end{equation*}

\noindent Consequently, for $j \in \{1,2\}$,
$$\left|\frac{\partial}{\partial t}\int_{\R}Q_{j}(\Vec{u}(t))\,dx\right|\leq  \frac{C |\dot{z}_j(t)|}{t^2 \log^2 t} ,$$
which shows the desired result.
\end{proof}
We will now prove that the operator $\mathcal{H}_{\operatorname{loc}}$ still satisfies the coercivity property.
\begin{proposition}\label{coercivity2}
				There exists $K>0$ such that
				\[\begin{aligned}
					\langle \mathcal{H}_{loc} \vec{\varepsilon},\vec{\varepsilon}\rangle &\geq K\|\vec{\varepsilon}\|_{X}^2-\sum_{k=1}^{2}(\lambda_{k}^{+})^{2}-\sum_{k=1}^{2}(\lambda_{k}^{-})^{2}
				\end{aligned}\]
			\end{proposition}
            \begin{proof}
				First, we give a localized version of Lemma \ref{coer}. Let $\Phi: \mathbb{R} \rightarrow \mathbb{R}$ be a $C^{2}$-function such that $\Phi(x)=\Phi(-x), \Phi^{\prime} \leqslant 0$ on $\mathbb{R}^{+}$with
				\[
				\begin{array}{ll}
					\Phi(x)=1 & \text { on }[0,1] ; \quad \Phi(x)=e^{-x} \quad \text { on }[2,+\infty) \\
					& e^{-x} \leqslant \Phi(x) \leqslant 3 e^{-x} \quad \text { on } \mathbb{R} .
				\end{array}
				\]
				
				Let $B>0$, and let $\Phi_{B}(x)=\Phi(x / B)$. Set
				\[
				\begin{aligned}
					\frac{1}{2}\langle \mathcal{H}_{\Phi_{B}}\vec{\varepsilon}, \vec{\varepsilon}\rangle&=\int_{\R}\Phi_{B}\left(-x_{0}\right)|\partial_{x}\varepsilon_1|^2\,dx+\int_{\R}\Phi_{B}\left(-x_{0}\right)|\varepsilon_1|^2\,dx+\int_{\R}\Phi_{B}\left(-x_{0}\right)|\varepsilon_2|^2\,dx\\&\quad+\frac{\alpha}{2}\int_{\R}\Phi_{B}\left(-x_{0}\right)|\varepsilon_3|^2\,dx+\frac{\alpha}{2}\int_{\R}\Phi_{B}\left(-x_{0}\right)|\varepsilon_4|^2\,dx+\frac{\beta}{2}\int_{\R}\Phi_{B}\left(-x_{0}\right)|\varepsilon_1|^4\,dx\\
    &\quad + \alpha \int_{\R}|\varepsilon_1|^2R_{0}^{(3)}\,dx+ 2\alpha \int_{\R} R_{0}^{(1)}\varepsilon_3\overline{\varepsilon}_1\,dx\\&\quad +2 \omega_0 \int_{\R}\Phi_{B}\left(-x_{0}\right) \varepsilon_2 \overline{\varepsilon}_1\,dx \quad + 2 \beta \int_{\R} \Re(R_0^{(1)} \overline{\varepsilon}_1)R_0^{(1)} \overline{\varepsilon}_1\,dx.
				\end{aligned}
		\]

				For the sake of simplicity, we assume that $x_{0}=0$. We set $z_1=\varepsilon_1\sqrt{\Phi_{B}}$ , $z_2=\varepsilon_2\sqrt{\Phi_{B}}$, $z_3=\varepsilon_3\sqrt{\Phi_{B}}$ and $z_4=\varepsilon_4\sqrt{\Phi_{B}}$. Then, by simple calculations,
\[
 	\int_{\R}\left|\partial_{x} \varepsilon_1\right|^{2} \Phi_{B}\,dx=\int_{\R}\left|\partial_{x} z_1\right|^{2}\,dx+\frac{1}{4} \int|z_1|^{2}\left(\frac{\Phi_{B}^{\prime}}{\Phi_{B}}\right)^{2}\,dx-2 \int_{\R} \partial_{x} z_1 \bar{z}_1 \frac{\Phi_{B}^{\prime}}{\Phi_{B}}\,dx,\]
  $$\int_{\R}\Phi_{B}\varepsilon_2 \partial_{x}\overline{\varepsilon}_1\,dx=\int_{\R}z_2 \partial_{x}\overline{z}_1\,dx-\frac{1}{2}\int_{\R}z_2 \overline{z}_1 \frac{\Phi_{B}^{\prime}}{\Phi_{B}}\,dx$$
  \[
  \int_{\R}|\varepsilon_j|^{2} \Phi_{B}\,dx=\int_{\R}|z_j|^{2}\,dx,
  \quad j=1,2,3,4\] 
  \[
  \int_{\R}|\varepsilon_1|^{4} \Phi_{B}\,dx=\int_{\R}|z_1|^{2}\frac{1}{\Phi_{B}}\,dx,
  \] 
				and \[\int_{\R} \varepsilon_k \varepsilon_j \Phi_{B}\,dx=\int_{\R}z_j z_k\,dx. \quad j \neq k
\]
	 	Since, by definition of $\Phi_{B}$, we have $\left|\Phi_{B}^{\prime}\right| \leqslant(C / B) \Phi_{B}$, we obtain
 \[
 \begin{split}
 \int_{\R}\left|\partial_{x} z_1\right|^{2}\,dx-\frac{C}{B} \int_{\R}\left(\left|\partial_{x} z_1\right|^{2}+|z_1|^{2}\right) \,dx& \leqslant \int_{\R}\left|\partial_{x} \varepsilon_1\right|^{2} \Phi_{B}\,dx\\&
 \leqslant \int_{\R}\left|\partial_{x} z_1\right|^{2}\,dx+\frac{C}{B} \int_{\R}\left(\left|\partial_{x} z_1\right|^{2}+|z_1|^{2}\right)\,dx.
 \end{split}\]
 Also,  \[
  \int_{\R}|\varepsilon_1|^{4} \Phi_{B}\,dx\geq \int_{\R}|z_1|^{2}\,dx,
  \] 
 We also have
 $$\alpha \int_{\R}|\varepsilon_1|^2R_{0}^{(3)}\,dx=\alpha \int_{\R}|z_1|^2R_{0}^{(3)}\frac{1}{\Phi_{B}}\,dx,$$
\[
 \begin{aligned} 	2\alpha \int_{\R} R_{0}^{(1)}\varepsilon_3\overline{\varepsilon}_1\,dx=2\alpha \int_{\R} R_{0}^{(1)}z_3\overline{z}_1\frac{1}{\Phi_{B}}\,dx
				\end{aligned}
\]
and $$ 2 \beta \int_{\R} \Re(R_0^{(1)} \overline{\varepsilon}_1)R_0^{(1)} \overline{\varepsilon}_1\,dx= 2 \beta \int_{\R} \Re(R_0^{(1)} \overline{z}_1)R_0^{(1)} \overline{z}_1\frac{1}{\Phi_{B}}\,dx$$
								Since $\Phi_{B} \equiv 1$ on $[-B, B]$ and $\Phi_{\omega_{0}}^{(1)}(x) \leqslant C e^{-\left(\sqrt{\omega_{0}} / 2\right)|x|}$, we have, for all $x \in \mathbb{R}$,
				
\[
				\left|\frac{1}{\Phi_{B}}-1\right| \Phi_{\omega_{0}}^{(1)}(x) \leqslant e^{-\left(\sqrt{\omega_{0}}-2 / B\right)|x| / 2} \leqslant C e^{-\sqrt{\omega_{0}} B / 4} \leqslant \frac{1}{B}
\]
								for $B$ large enough. Thus,
\[
				\begin{aligned}
					\alpha \int_{\R}|\varepsilon_1|^2R_{0}^{(3)}\,dx\leqslant C\alpha \int_{\R}|z_1|^2R_{0}^{(3)}\,dx+\frac{C}{B} \int_{\R}|z_1|^{2}\,dx.
				\end{aligned}
\]
\[
 \begin{aligned} 	2\alpha \int_{\R} R_{0}^{(1)}\varepsilon_3\overline{\varepsilon}_1\,dx\leq C \int_{\R} R_{0}^{(1)}z_3\overline{z}_1\,dx+\frac{C}{B}\int_{\R} z_3\overline{z}_1\,dx
				\end{aligned}
\]
and $$ 2 \beta \int_{\R} \Re(R_0^{(1)} \overline{\varepsilon}_1)R_0^{(1)} \overline{\varepsilon}_1\,dx\leq C  \int_{\R} \Re(R_0^{(1)} \overline{z}_1)R_0^{(1)} \overline{z}_1\,dx+\frac{C}{B}\int_{\R} \Re( \overline{z}_1) \overline{z}_1\,dx.$$
Collecting these calculations, we obtain
\[
				\langle \mathcal{H}_{\Phi_{B}}\vec{\varepsilon}, \vec{\varepsilon}\rangle  \geqslant \langle \mathcal{H}_{0}{\vec{z}}, {\vec{z}}\rangle-\frac{C}{B} \int\left(\left|\partial_{x} z_1\right|^{2}+|z_1|^{2}+ |z_2|^{2}+|z_3|^{2}+|z_4|^{2}+|z_1|^{4}\right).
\]
								Thanks to the orthogonality conditions on ${\vec z}=(z_1,z_2,z_3,z_4)$, we verify easily using the property of $\Phi_{B}$ that
$ 
(\vec{z},\partial_{x}\vec{R})_{L^2(\R)}=0,
$ 
				for $B$ large enough. By Lemma \ref{coer}, we obtain  for $B$ large enough that
\[
				\begin{aligned}
   \langle \mathcal{H}_{\Phi_{B}}\vec{\varepsilon}, \vec{\varepsilon}\rangle  & \geqslant\left(C-\frac{C}{B}\right)\|{\vec{z}}\|_{X}^{2}+ \sum_{k=1}^{2}(\lambda_{k}^{+})^{2}+\sum_{k=1}^{2}(\lambda_{k}^{-})^{2}\\& \geqslant \frac{C}{2}\|{\vec{z}}\|_{X}^{2} + \sum_{k=1}^{2}(\lambda_{k}^{+})^{2}+\sum_{k=1}^{2}(\lambda_{k}^{-})^{2}\\&\geqslant \frac{C}{2}\left(1-\frac{C}{B}\right) \int_{\R}\left(\left|\partial_{x} z_1\right|^{2}+|z_1|^{2}+ |z_2|^{2}+|z_3|^{2}+|z_4|^{2}+|z_1|^{4}\right) \Phi_{B}\,dx \\& + \sum_{k=1}^{2}(\lambda_{k}^{+})^{2}+\sum_{k=1}^{2}(\lambda_{k}^{-})^{2}\\
					& \geqslant \frac{C}{4} \int_{\R}\left(\left|\partial_{x} z_1\right|^{2}+|z_1|^{2}+ |z_2|^{2}+|z_3|^{2}+|z_4|^{2}+|z_1|^{4}\right) \Phi_{B}\,dx\\&+ \sum_{k=1}^{2}(\lambda_{k}^{+})^{2}+\sum_{k=1}^{2}(\lambda_{k}^{-})^{2},
				\end{aligned}
\]
				implying \begin{equation*}\label{claim1}
					\langle \mathcal{H}_{\Phi_{B}}\vec{\varepsilon}, \vec{\varepsilon}\rangle \geqslant C\|\vec{\varepsilon}\|_{X}^{2}+ \sum_{k=1}^{2}(\lambda_{k}^{+})^{2}+\sum_{k=1}^{2}(\lambda_{k}^{-})^{2}.   
				\end{equation*}
 Now, let $B>B_{0}$, and   $L>0$. Since $\sum_{k=1}^{2} \phi_{k}(t) \equiv 1$, we decompose $\langle \mathcal{H}\vec{\varepsilon}, \vec{\varepsilon}\rangle$ as follows:
\[
				\begin{aligned}
					&\langle \mathcal{H}_{loc}\vec{\varepsilon}, \vec{\varepsilon}\rangle=\\ & \sum_{k=1}^{2} \int_{\R} \Phi_{B}\left(-z_{k}(t)\right)\bigg\{\left|\partial_{x} \varepsilon_1\right|^{2}+|\varepsilon_1|^{2}+\frac{\beta}{2}|\varepsilon_1|^{4}+|\varepsilon_2|^{2}+\frac{\alpha}{2}|\varepsilon_3|^{2}+\frac{\alpha}{2}|\varepsilon_4|^{2}+2\omega_k \overline{\varepsilon}_1 \varepsilon_2\\
   & \quad + \alpha \int_{\R}|\varepsilon_1|^2R_{k}^{(3)}\,dx+ 2\alpha \int_{\R} R_{k}^{(1)}\varepsilon_3\overline{\varepsilon}_1\,dx+2 \beta \int_{\R} \Re(R_k^{(1)} \overline{\varepsilon}_1)R_k^{(1)} \overline{\varepsilon}_1\bigg\}\,dx\\
					& +\sum_{k=1}^{2}\int_{\R}\left(\phi_{k}-\Phi_{B}\left(-z_{k}(t)\right)\right)\bigg\{\left|\partial_{x} \varepsilon_1\right|^{2}+|\varepsilon_1|^{2}+\frac{\beta}{2}|\varepsilon_1|^{4}+|\varepsilon_2|^{2}+\frac{\alpha}{2}|\varepsilon_3|^{2}+\frac{\alpha}{2}|\varepsilon_4|^{2}-2\omega_k \overline{\varepsilon}_1 \varepsilon_2\bigg\}\,dx.
				\end{aligned}
	\]
				By \eqref{claim1}, for any $k=1,2$, we have  for $B$ large enough that
\[
				\begin{aligned}
			 &\int_{\R} \Phi_{B}\left(-x_k(t)\right)		\bigg\{\left|\partial_{x} \varepsilon_1\right|^{2}+|\varepsilon_1|^{2}+\frac{\beta}{2}|\varepsilon_1|^{4}+|\varepsilon_2|^{2}+\frac{\alpha}{2}|\varepsilon_3|^{2}+\frac{\alpha}{2}|\varepsilon_4|^{2}-2\omega_k \overline{\varepsilon}_1 \varepsilon_2\\
     &\quad -\alpha c_{k}\varepsilon_3 \varepsilon_4-2c_{k}\partial\overline{\varepsilon}_1 \varepsilon_2\bigg\}\,dx   + \alpha \int_{\R}|\varepsilon_1|^2R_{k}^{(3)}\,dx+ 2\alpha \int_{\R} R_{k}^{(1)}\varepsilon_3\overline{\varepsilon}_1\,dx+2 \beta \int_{\R} \Re(R_k^{(1)} \overline{\varepsilon}_1)R_k^{(1)} \overline{\varepsilon}_1\,dx
					\\&\geqslant \lambda_{k} \int_{\R}\Phi_{B}\left(-z_{k}(t)\right)\left(\left|\partial_{x} \varepsilon_1\right|^{2}+|\varepsilon_2|^{2}+|\varepsilon_1|^{2}+|\varepsilon_3|^{2}+|\varepsilon_4|^{2}\right)\,dx.
				\end{aligned}
\]
				Moreover, by the properties of $\Phi_{B}$ and $\phi_{k}(t)$, for $L$ large enough, we have
				
\[
				\phi_{k}(t)-\Phi_{B}\left(-x_{k}(t)\right) \geqslant-e^{-L /(4 B)},
\]
	and  for
	 $\delta_{k}=\delta_{k}\left(c_k,\omega_{k}\right)>0$,
\[
		\begin{aligned}
		  &  \left|\partial_{x} \varepsilon_1\right|^{2}+|\varepsilon_1|^{2}+\frac{\beta}{2}|\varepsilon_1|^{4}+|\varepsilon_2|^{2}+\frac{\alpha}{2}|\varepsilon_3|^{2}+\frac{\alpha}{2}|\varepsilon_4|^{2}-2\omega_k \overline{\varepsilon}_1 \varepsilon_2 \\&\geqslant \delta_{k}\left(\left|\partial_{x} \varepsilon_1\right|^{2}+|\varepsilon_1|^{2}+|\varepsilon_2|^{2}+|\varepsilon_3|^{2}+|\varepsilon_4|^{2}\right) \geqslant 0.
		\end{aligned}		
\]
				So,
\[
				\begin{aligned}
					& \int_{\R}\left(\phi_{k}-\Phi_{B}\left(-z_k(t)\right)\right)\bigg\{\left|\partial_{x} \varepsilon_1\right|^{2}+|\varepsilon_1|^{2}+\frac{\beta}{2}|\varepsilon_1|^{4}+|\varepsilon_2|^{2}+\frac{\alpha}{2}|\varepsilon_3|^{2}+\frac{\alpha}{2}|\varepsilon_4|^{2}-2\omega_k \overline{\varepsilon}_1 \varepsilon_2\bigg\}\,dx\\
     & \geqslant \delta_{k} \int_{\R}\left(\varphi_{k}(t)-\Phi_{B}\left(-x_{k}(t)\right)\right)\left(\left|\partial_{x} \varepsilon_1\right|^{2}+|\varepsilon_1|^{2}+|\varepsilon_2|^{2}+|\varepsilon_3|^{2}+|\varepsilon_4|^{2}\right) \,dx\\
					&\geqslant-C e^{-L /(4 B)} \int_{\R}\left(\left|\partial_{x} \varepsilon_1\right|^{2}+|\varepsilon_1|^{2}+|\varepsilon_2|^{2}+|\varepsilon_3|^{2}+|\varepsilon_4|^{2}\right) \,dx.
				\end{aligned}
\]
    Thus, our above considerations reveal that 
\[
				\begin{aligned}
				&\langle \mathcal{H}_{loc}\vec{\varepsilon}, \vec{\varepsilon}\rangle+ \sum_{k=1}^{2}(\lambda_{k}^{+})^{2}+\sum_{k=1}^{2}(\lambda_{k}^{-})^{2}\\&\geqslant K \int_{\R}\left(\sum_{k=1}^{2} \phi_{k}\right)\left(\left|\partial_{x} \varepsilon_1\right|^{2}+|\varepsilon_1|^{2}+|\varepsilon_2|^{2}+|\varepsilon_3|^{2}+|\varepsilon_4|^{2}\right)\,dx\\
     &\quad-C e^{-L /(4 B)} \int_{\R}\left(\left|\partial_{x} \varepsilon_1\right|^{2}+|\varepsilon_1|^{2}+|\varepsilon_2|^{2}+|\varepsilon_3|^{2}+|\varepsilon_4|^{2}\right)\,dx  
				\end{aligned}
\]
				
				and since $\sum_{k=1}^{2} \phi_{k}(t) \equiv 1$, we obtain the result by taking $L$ large enough.
			\end{proof}
 \begin{lemma}
       \label{Bootstrap3}
There exist $T_{0}\in \R$   and $n_{0} \in\N$ such that, if for all $n \geq n_{0}$, every approximate solution $\Vec{u}^n$ is defined on $[T_{0},T^{n}]$, and for all $t \in[t_{0},T^{n}]$, with $t_{0}\in [T_{0},T^{n}]$,
\begin{equation*}
\begin{aligned}
&\left| (\kappa g_0 )^{-1/2}e^{\frac{|z(t)|}{2}}-t\right|\leq t^{-1}\log^{-1/2}  t, \, \, \quad\|\vec{\varepsilon}\|_{X}\leq  \frac{1}{t \log^{3/2} t}\\
& \sum_{k=1,2}\left|\lambda_{k}^{+}(t)\right| \leq t^{-3 / 2}    \, \, \, \sum_{k=1,2}\left|\lambda_{k}^{-}(t)\right| \leq t^{-3 / 2}.
\end{aligned}
\end{equation*}
then for all $t\in[t_{0},T^{n}]$, 
\begin{equation*}
  \begin{aligned}
& \|\vec{\varepsilon}\|_{X}\leq \frac{1}{2t \log^{3/2} t} \quad \quad \left| (\kappa g_0 )^{-1/2}e^{\frac{|z(t)|}{2}}-t\right|\leq \frac{1}{2}\log^{-1/2}  t, 
\\&\sum_{k=1,2}\left|\lambda_{k}^{+}(t)\right| \leq \frac{1}{2}t^{-3 / 2}, 
\sum_{k=1,2}\left|\lambda_{k}^{-}(t)\right| \leq \frac{1}{2}t^{-3 / 2}.
\end{aligned}  
\end{equation*}
   \end{lemma}
\begin{proof}
    We start with the estimate of $\vec{\varepsilon}$. For this we use the Proposition \ref{VariacionS} and Lemma \ref{coercivity2}. First, we can observe that by integrating over   $[t,T_n]$ in Proposition \ref{VariacionS} and from remark \ref{remarklast}, it follows that $$|S(\vec{u}(t))-S(\vec{R}(t))|=|S(\vec{u}(t))-S(\vec{u}(T_n))|\leq C_0t^{-2}\log^{-3} t.$$

Thus by the Lemma \ref{newS}  and Lemma \ref{coercivity2}, we have,
$$
\|\vec{\varepsilon}\|_{X}^{2} \leq C_{0}t^{-2}\log^{-3} t+C_{1} t^{-3}
$$
Therefore for large enough $T_{0}$ depending on $C_{0}$ and $C_{1}$ we have
$$
\|\vec{\varepsilon}\|_{X}^{2} \leq  C t^{-2} \log ^{-3} t.
$$
 Then the estimate on $\vec{\varepsilon}$ can be strictly improved.

Now we are going to estimate on $\lambda_{k}^{+}$. Using the Proposition \ref{direction} and the initial data in \eqref{condlast}, we get
$$
\left|\lambda_{k}^{+}(t)\right| \lesssim C e^{\epsilon_{0}^k t} \int_{t}^{T_{n}} e^{-\epsilon_{0}^k \tau} \tau^{-2}\, d \tau \leq \frac{C^{2}}{\epsilon_{0}^k} t^{-2} \leq \frac{1}{4} t^{-3 / 2}.
$$
for $T_{0}$ to be large enough.
\end{proof}
Next, we use a topological argument to close the bootstrap on the parameters $z$ and $\lambda_{k}^{-}$ for $k=1,2$. We will show that there exists a choice of initial data, $\left|\lambda_{k}\right|\leq T_{n}^{-3 / 2}$ 
and $\bar{z}>0$ such that $T^{\star}=T_{0}$. This argument will thus conclude the proof of Lemma \ref{Bootstrap3}.

\begin{lemma}\label{T0}
   There exists $n_{0} \geq 1$ large enough, such that for $n \geq n_{0}$ there exists $\bar{z}_{n}>0$ and $|a_{k, n} | \leq T_{n}^{-3 / 2}$ 
   for $k=1,2$ such that $T^{\star}=T_{0}$. 
\end{lemma} 
\begin{proof}
We will use the method of reduction by contradiction. Let
$$
A(t)=\left(\kappa g_{0}\right)^{-1 / 2} e^{\frac{1}{2}|z(t)|}, \quad and \quad D(t)=(A(t)-t)^{2} t^{-2} \log t
$$
and we denote by 
$$
L(t)=\sum_{k=1,2} t^{3}\left|\lambda_{k}^{-}(t)\right|^{2}=t^3|\lambda^{-}(t)|^2.
$$

Suppose that for all $(\hat{\xi}, \hat{a})=\left(\hat{\xi}, \hat{a}_{1}, \hat{a}_{2}\right) \in \mathbb{D}=[-1,1] \times \mathcal{B}(1)$ the choice,
 {where  $\mathcal{B}(1)$ is the zero-centered ball of radius 1 in $\R^2$},
\begin{equation}\label{condlast2}
 A\left(T_{n}\right)=T_{n}+\hat{\xi} T_{n} \log ^{-1 / 2} T_{n}, \quad \lambda^{-}\left(T_{n}\right)=\hat{a}\left(T_{n}\right)^{-3 / 2}   
\end{equation}
gives us $T^{\star}=T^{\star}(\hat{\xi}, \hat{a}) \in\left(T_{0}, T_{n}\right]$. We will start by showing the estimate $$\left |\left(\kappa g_{0}\right)^{-1 / 2} e^{\frac{1}{2}|z(t)|}-t\right|\leq \frac{1}{2}t^{-1}\log ^{-1/2} t.$$ 
Then, observe that

\begin{equation}\label{ness2}
\dot{D}(t)=2(A(t)-t)(\dot{A}(t)-1) t^{-2} \log t-(A(t)-t)^{2}\left(2 t^{-3} \log t-t^{-3}\right).
\end{equation}

Before proceeding to estimate $\dot{D}$ we first claim that for all $t \in\left(T^{\star}(\hat{\xi}, \hat{a}), T_{n}\right]$,

\begin{equation}\label{ness1}
\left|\frac{\mathrm{d}}{\mathrm{~d} t}\left( e^{\frac{1}{2}|z|}\right)-\sqrt{\kappa g_{0}}\right| \lesssim t^
{-1}\log ^{-1} t .
\end{equation}

Fixed $k \in \{1,2\}$. We know from the triangular inequality and \eqref{newtheta}
$$
\left|\dot{\theta}_k \frac{z}{|z|}+2 \kappa g_{0}e^{-|z|} \right| \lesssim t^{-2} \log ^{-1} t,
$$
Similarly usingwe get $|\dot{z}_k| \lesssim t^{-1} \log ^{-1} t$ which in turn implies that
$$
\left|\dot{z}_k  \frac{z}{|z|}-\theta_k  \frac{z}{|z|}\right| \lesssim t^{-1} \log ^{-1} t
$$

Therefore using Lemma \ref{modu1} and the estimates above,
$$
\left|\left(\dot{\theta}_k \frac{z}{|z|}\right)\left(\theta_k  \frac{z}{|z|}\right)+2 \kappa g_{0} \dot{z}_k  \frac{z}{|z|} e^{-|z|} \right|\lesssim t^{-3} \log ^{-1} t .
$$

Thus using the Proposition \ref{prop1} we can integrate on the interval $\left[t, T_{n}\right]$ where $t \in\left[T^{\star}(\hat{\zeta}, \hat{a}), T_{n}\right]$ to get 
$$
\begin{aligned}
\left|\frac{1}{2}\left(\theta_k  \frac{z}{|z|}\right)^{2}-2 \kappa g_{0} e^{-|z|}\right| \lesssim t^{-2} \log ^{-1} t
\end{aligned}
$$

Thus we have,
$$
\left|\left(\theta_k  \frac{z}{|z|}\right)-\sqrt{4 \kappa g_{0}}e^{-\frac{1}{2}|z|}\right|+\left|\left(\dot{z}_k  \frac{z}{|z|}\right)-\left(\theta_k  \frac{z}{|z|}\right)\right| \lesssim t^{-1} \log ^{-1} t.
$$
which implies,
\begin{equation}\label{deriv}
 \left|\left(\dot{z}  \frac{z}{|z|}\right)-\sqrt{4 \kappa g_{0}}e^{-\frac{1}{2}|z|} \right |\lesssim t^{-1} \log ^{-1} t.   
\end{equation}
Now, since
$$
\frac{\mathrm{d}}{\mathrm{~d} t}\left(e^{\frac{1}{2}|z|}\right)=\frac{1}{2} \dot{z}  \frac{z}{|z|} e^{\frac{1}{2}|z|} .
$$

Thus, from \eqref{deriv}
$$
\left|\frac{\mathrm{d}}{\mathrm{~d} t}\left( e^{\frac{1}{2}|z|}\right)-\sqrt{\kappa g_{0}}\right| \lesssim t^{-1} \log ^{-1}(t).
$$
Hence, 
\begin{equation}\label{ness3}
|\dot{A}(t)-1| \lesssim t^
{-1}\log ^{-1} t .
\end{equation}
Which implies from \eqref{ness2} that  $\dot{D}$ is limited, then applying integration on $[t,T_n]$ for $t\in [T^{\star},T_n] $ and  using \eqref{condlast2} the desired result. On the other hand, to estimate $\dot{\mathcal{L}}$ we use \eqref{estiprop1}, \eqref{hipotese1} and Lemma \ref{direction}, to get that for all $t \in\left(T^{\star}(\hat{\xi}, \hat{a}), T_{n}\right]$
$$
\begin{aligned}
\dot{L}(t) & =\sum_{k=1,2} t^{3}\left(3 t^{-1} \lambda_{k}^{-}(t)+2 \frac{d \lambda_{k}^{-}}{d t}(t)\right) \lambda_{k}^{-}(t) \\
& =\sum_{k=1,2} t^{3}\left(3 t^{-1}-2 \epsilon_{0}^k\right) \lambda_{k}^{-}(t)^{2}+O\left(\|\vec{\varepsilon}\|_{X}^{2} t^{3}\left|\lambda_{k}^{-}(t)\right|+t\left|\lambda_{k}^{-}(t)\right|\right)  \\
& \leq\left(3 t^{-1}-2 \epsilon_{0}\right) L(t)+C t^{-1 / 2}\left(C^2 \log ^{-3}(t)+1\right) \sqrt{L(t)}
\end{aligned}
$$
where $\epsilon_0=\min\{\epsilon_0^1,\epsilon_0^2\}$ and  $C>0$ is a constant. Thus for $T_{0}$ large enough depending on $C$,
\begin{equation}\label{ness4}
\dot{L}(t) \leq-\frac{3 \epsilon_{0}}{2} L(t)+\frac{\epsilon_{0}}{2} \sqrt{L(t)} 
\end{equation}
Let,
$$
\Psi_{1}(t)=(A(t)-t) t^{-1} \log ^{1 / 2} t, \quad \Psi_{2}(t)=\left(\lambda_{1}^{-}(t) t^{3 / 2}, \lambda_{2}^{-}(t) t^{3 / 2}\right)=\lambda^{-}(t)\, t^{3 / 2}
$$

From the definition of $T^{\star}$ and continuity of the flow, at the limit $T^{\star}(\hat{\xi}, \hat{a})$ we have one of the following situation
$$
\begin{aligned}
\Psi_{1}\left(T^{\star}(\hat{\xi}, \hat{a})\right) & = \pm 1, \quad \Psi_{2} \in \mathcal{B}(1) \\
\left|\Psi_{2}\left(T^{\star}(\hat{\xi}, \hat{a})\right)\right| & =1 \, \, \text{if and only if}\, \,  \Psi_{2} \in \partial \mathcal{B}(1), \quad \Psi_{1} \in[-1,1],
\end{aligned}
$$
where  $\partial \mathcal{B}(1)$ is the boundary of the closed ball $\mathcal{B}(1)$. Now, using  \eqref{ness2} and \eqref{ness3} we get,
$$
\left|\dot{D}\left(T^{\star}(\hat{\xi}, \hat{a})\right)+2\left(T^{\star}(\hat{\xi}, \hat{a})\right)^{-1}\right| \lesssim\left(T^{\star}(\hat{\xi}, \hat{a})\right)^{-1} \log ^{-1 / 2}\left(T^{\star}(\hat{\xi}, \hat{a})\right)
$$
and so
$$
\dot{D}\left(T^{\star}(\hat{\xi}, \hat{a})\right)<-T^{\star}(\hat{\xi}, \hat{a})^{-1}<0
$$
for large enough $T_{0}$  whereas in the second case we have that $L\left(T^{\star}(\hat{\xi}, \hat{a})\right)=1$ and so from \eqref{ness4} we get,
$$
\dot{L}\left(T^{\star}(\hat{\xi}, \hat{a})\right) \leq-\frac{1}{2} \epsilon_{0}<0
$$
This transversality property implies the continuity of the map $(\hat{\xi}, \hat{a}) \rightarrow T^{\star}((\hat{\xi}, \hat{a}))$ and hence the following map,
$$
\begin{aligned}
& \Psi: \mathbb{D} \rightarrow \partial \mathbb{D} \\
& (\hat{\xi}, \hat{a}) \rightarrow\left(\Psi_{1}\left(T^{\star}(\hat{\xi}, \hat{a})\right), \Psi_{2}\left(T^{\star}(\hat{\xi}, \hat{a})\right)\right.
\end{aligned}
$$
is also continuous where $\partial \mathbb{D}$ denotes the boundary of $\mathbb{D}$. Note that if $\hat{a} \in \partial \mathcal{B}(1)$ then \eqref{ness4} implies that $\dot{L}\left(T_{n}\right)<0$, hence we have $T^{\star}(\hat{\xi}, \hat{a})=T_{n}$ and if $\hat{\xi}= \pm 1$ then from \eqref{ness3} we get $\dot{\xi}\left(T_{n}\right)<0$ and $T^{\star}(\hat{\xi}, \hat{a})=T_{n}$. Thus $\Psi(\hat{\xi}, \hat{a})=(\hat{\xi}, \hat{a})$ for all $(\hat{\xi}, \hat{a}) \in \partial \mathbb{D}$. Therefore the restriction of $\Psi$ to the boundary of $\mathbb{D}$ is the identity. However this contradicts the no retraction theorem. Thus there exists initial data ( $\bar{z}, a_{k}$ ) for $k=1,2$ such that $T^{\star}=T_{0}$.
\end{proof}
\begin{remark}
    Note that from Lemma \ref{T0} combined with Lemma \ref{Bootstrap3} the Proposition clearly follows \ref{prop1}. Furthermore, consider $n \in \N$ with $n \geq n_{0}$. As we will see, the argument below shows that as long as the approximate solution $\mathbf{u}^{n}$ exists, it satisfies \eqref{estiprop1}, which in turn implies that it does not explode in finite time. Consequently, we can assume that is defined on $[T_{0},T^{n}]$ and $\vec{u}_{n} \in \mathbf{C}([T_{0},T^{n}],X ).$
\end{remark}
\section*{Existence of Dipole}
In this final section, our goal is to guarantee the existence of dipole solutions for system \eqref{system1}. To that end, we first establish an immediate lemma before proceeding with our analysis. 
\begin{lemma}\label{companess}
Let $\vec{u}_{n} \in \mathcal{C}\left([T_0, T],X\right)$ be a sequence of solutions to \eqref{system1} and assume that for some $C>0$,
$$
\vec{u}_{n}(T_0) \rightharpoonup \vec{u}^{0} \text { in } X-\text { weak, and } \quad \forall n, \quad\left\|\vec{u}_{n}(t)\right\|_{\mathcal{C}\left([T_0, T], X\right)} \leq C.
$$
\end{lemma} 
Once the previous lemma has been established, we have all the necessary tools to provide a proof of our main result. 
\begin{proof}[Proof of Theorem \ref{2maintheorem}]
 From Proposition \ref{prop1} there exists a sequence of final data functions $\vec{u}_{n} \in X$ such that,

%$$
%\forall t \in\left[T_{0}, T_{n}\right], \quad \vec{u}_{n}(t)=\Phi\left(T_{n}, t, \vec{u}_{0, n}\right)
%$$
%where $\Phi$ is the flow of \eqref{system1}. Note that $T_{0}$ does not depend on $T_{n}$ and that there exists $C>0$ independent of $n$ such that
$$
\forall t \in\left[T_{0}, T_{n}\right], \quad\left\|\vec{u}_{n}(t)-\vec{R}_{1,n}(t)-\vec{R}_{2,n}(t)\right\|_{X} \leq C t^{-1} \log ^{-3 / 2} t .
$$

Let $\vec{u}_{0}$ be a weak limit in $X$ of the bounded sequence $\vec{u}_{n}\left(T_{0}\right)$ up to a subsequence extraction. Thus, from the well-posed locally of the system \eqref{system2} in $X$ we have that there exists $\vec{u} \in C([T_0,T^{\star}],X)$ such that
$
\vec{u}(T_0)=\vec{u}_{0}$, where $T^{\star}$ is the maximum time of existence of the solution.

Now, using the Lemma \ref{companess} on $\left[T_{0}, T^{\star}\right)$ we have $$\vec{u}_{n}(T_0) \rightharpoonup \vec{u}_0 \, \, \text{weakly in} \, \,  X$$
then since there continued dependence of $\vec{u}$ on $\left[T_{0}, T^{\star}\right)$ it follows that $$\vec{u}_{n}(t) \rightharpoonup \vec{u}(t) \, \, \text{weakly in} \, \, X.$$ 
Therefore, $$
\forall t \in\left[T_{0}, T^{\star}\right), \quad\left\|\vec{u}(t)-\vec{R}_1(t)-\vec{R}_2(t)\right\|_{X} \leq C t^{-1} \log ^{-3 / 2} t .
$$
From where we deduced that $T^{\star}=\infty$, it could not be so that we would obtain a contradiction of the alternative blow up generated from the well-posed locally of the system.
 Moreover, note that the estimates of the Lemma \ref{modu1} provide uniform bounds on the time derivatives of the geometric parameters $\left(z_{k, n}(t), \theta_{k, n}(t)\right)_{k=1,2}$ on the interval $\left[T_{0}, T_{n}\right]$. Therefore by Ascoli Lemma we get uniform convergence as $n \rightarrow+\infty$,

$$
\left(z_{k, n}(t), \theta_{k, n}(t)\right)_{k=1,2} \rightarrow\left(z_{k}(t), \theta_{k}(t)\right)_{k=1,2}
$$
on compact subsets of $\left[T_{0},+\infty\right)$ up to a subsequence extraction for some continuous functions $\left(z_{k}(t), \theta_{k}(t)\right)_{k=1,2}$. Thus as $n \rightarrow+\infty$,
$$
\vec{\varepsilon}_{n}(t) \rightharpoonup \vec{\varepsilon}(t) \text { in } X \text { - weak }
$$
for $t \in\left[T_{0},+\infty\right)$. Therefore for $t \in\left[T_{0},+\infty\right)$, using the uniform estimates due to Proposition \ref{prop1} we get,
\begin{equation*}
||z(t)|-2 \log t| \lesssim \log^{-1/2} t,  \quad\|\vec{\varepsilon}(t)\|_{X} \lesssim t^{-1} \log ^{-3 / 2} t.
\end{equation*}   
Thus, by applying the limit step we reach the desired result.
\end{proof}

			%%%%%%%%%%%%%%%%%%%%%%%%%%%%%%%%%%%%%%%%%%%%%

			\section*{Conflict of interest} The authors declare that they have no conflict of interest. 
			
			\section*{Data Availability}
			There is no data in this paper.

		\end{document}